\documentclass[12pt]{amsart}
\usepackage{amsmath, amsbsy,amssymb,amsbsy,amsfonts,amsthm,latexsym,
                        amsopn,amstext,amsxtra,euscript,amscd,mathrsfs,color,bm}
                        
\usepackage{float} 
\usepackage[english]{babel}
\usepackage{url}
\usepackage{todonotes}
\usepackage[colorlinks,linkcolor=blue,anchorcolor=blue,citecolor=blue,backref=page]{hyperref}
\usepackage[top=3 cm, bottom=3 cm, left=3 cm, right=3cm]{geometry}
    
\newtheorem{theorem}{Theorem}
\newtheorem{lemma}[theorem]{Lemma}

\newtheorem{corollary}[theorem]{Corollary}

\theoremstyle{definition}
\newtheorem{definition}[theorem]{Definition}

\numberwithin{equation}{section}
\numberwithin{theorem}{section}
\numberwithin{table}{section}
\numberwithin{figure}{section}

\allowdisplaybreaks

\definecolor{olive}{rgb}{0.3, 0.4, .1}
\definecolor{dgreen}{rgb}{0.,0.5,0.}
\definecolor{plum}{rgb}{0.6,0.,0.6}

\def\cA{{\mathcal A}}
\def\cB{{\mathcal B}}
\def\cC{\mathcal{C}}
\def\cD{{\mathcal D}}
\def\cE{\mathcal{E}}
\def\cS{{\mathcal S}}

\def\C{\mathbb{C}}
\def\F{\mathbb{F}}
\def\K{\mathbb{K}}
\def\Z{\mathbb{Z}}
\def\R{\mathbb{R}}
\def\Q{\mathbb{Q}}

\def\G{\mathbb{G}}

\def\trdeg{\mathrm{trdeg}}

\def\vec#1{\mathbf{#1}}
\def\ov#1{{\overline{#1}}}

\def\Gm{\G_{\textup{m}}}

\def\house#1{{%
    \setbox0=\hbox{$#1$}
    \vrule height \dimexpr\ht0+1.4pt width .5pt depth \dp0\relax
    \vrule height \dimexpr\ht0+1.4pt width \dimexpr\wd0+2pt depth \dimexpr-\ht0-1pt\relax
    \llap{$#1$\kern1pt}
    \vrule height \dimexpr\ht0+1.4pt width .5pt depth \dp0\relax}}

\def\({\left(}
\def\){\right)}

\def\rf#1{\left\lceil#1\right\rceil}

\def\mand{\qquad\mbox{and}\qquad}

\newcommand{\h}{\mathrm{h}}
\newcommand{\tor}{\mathrm{tor}}

\newcommand{\ord}{\mathrm{ord}}
\newcommand{\Res}{\mathrm{Res}}
\newcommand{\bphi}{\pmb{\varphi}}
\newcommand{\brho}{\pmb{\varrho}}

\def\lg{\left\lbrace}
\def\rg{\right\rbrace}
\newcommand{\ol}[1]{\mathbf{#1}}

\newcommand{\Qbar}{\overline{\Q}}
\newcommand{\abs}[1]{\left|#1\right|}

\newcommand{\End}{\text{End}}

\usepackage{xcolor}
\definecolor{darkGreen}{RGB}{0, 128, 0}

\begin{document}

\title[Multiplicative  and and linear dependence]{Multiplicative and linear dependence  in finite fields and on elliptic curves  modulo primes}

\author[Barroero]{Fabrizio Barroero}
\address{Dipartimento di Matematica e Fisica, Universit\`a di Roma Tre, Largo San Mu\-rial\-do 1, 00146 Roma, Italy}
\email{fbarroero@gmail.com}

\author[Capuano]{Laura Capuano}
\address{DISMA ``Luigi Lagrange", Politecnico di Torino, Corso Duca degli Abruzzi 24, 10129 Torino, Italy}
\email{laura.capuano@polito.it}

\author[M{\'e}rai]{L{\'a}szl{\'o} M{\'e}rai}
\address{Johann Radon Institute for Computational and Applied Mathematics, Altenberger Strasse 69, 4040 Linz, Austria}
\email{laszlo.merai@oeaw.ac.at}

\author[Ostafe]{Alina Ostafe}
\address{School of Mathematics and Statistics, University of New South Wales, Sydney, NSW 2052, Australia}
\email{alina.ostafe@unsw.edu.au}

\author[Sha]{Min Sha}
\address{School of Mathematical Sciences, South China Normal University, Guangzhou, 510631, China}
\email{shamin@scnu.edu.cn}

\subjclass[2010]{11T30, 11G05, 11G20, 11U09}

\keywords{Multiplicative dependence, linear dependence, rational function, elliptic curve, finite field, unlikely intersection, o-minimality} 

\begin{abstract} 
For positive integers $K$ and $L$, we introduce and study the notion of $K$-multiplicative dependence over the algebraic closure $\ov{\F}_p$ of a finite prime field $\F_p$, as well as $L$-linear dependence of points on elliptic curves in reduction modulo primes. One of our main results  shows that, given non-zero rational functions $\varphi_1,\ldots,\varphi_m, \varrho_1,\ldots,\varrho_n\in\Q(X)$ and an elliptic curve $E$ defined over the rational numbers $\Q$, for any sufficiently large prime $p$, for all but finitely many $\alpha\in\ov\F_p$, at most one of the following two can happen: $\varphi_1(\alpha),\ldots,\varphi_m(\alpha)$ are $K$-multiplicatively dependent or the points $(\varrho_1(\alpha),\cdot), \ldots,(\varrho_n(\alpha),\cdot)$ are $L$-linearly dependent on the reduction of $E$ modulo $p$. As one of our main tools, we prove a general statement about the intersection of an irreducible curve in the split semiabelian variety $\Gm^m \times E^n$ with the algebraic subgroups of codimension at least $2$.

As an application of our results, we improve a result of  M. C. Chang and extend a result of J. F. Voloch about elements of large order in finite fields in some special cases. 
\end{abstract}

\maketitle

\tableofcontents

\section{Introduction}
Let $\K$ be a field. We say that non-zero $\alpha_1,\ldots,\alpha_n\in \K$ are \textit{multiplicatively independent} 
if there is no non-zero integer vector $(k_1,\ldots,k_n) \in \Z^n$ such that 
\begin{equation}  \label{eq:MultDep}
\alpha_1^{k_1} \cdots \alpha_n^{k_n}=1. 
\end{equation}
If this is not the case, then they are called \textit{multiplicatively dependent}. 
In particular, this definition applies to rational functions as well.

Moreover, we say that non-zero rational functions $f_1,\ldots, f_n \in \K(X)$ are   \textit{multiplicatively independent modulo constants} 
if there is no non-zero integer vector $(k_1,\ldots,k_n)$ such that 
$$
f_1^{k_1} \cdots f_n^{k_n} \in \K^* : = \K \setminus \{0\}.
$$

Multiplicative dependence of algebraic numbers and of rational functions has been extensively studied in recent years from various aspects; see, for instance, \cite{BS,BOSS,BMZ, DS, OSSZ1,OSSZ2,PSSS,SSS}. In particular, a result of Bombieri, Masser and Zannier~\cite{BMZ} in the context of unlikely intersections over tori says that, given $n$ non-zero multiplicatively independent modulo constants rational functions $f_1,\ldots, f_n \in \Qbar(X)$, there are at most finitely many $\alpha \in \ov\Q$ such that $f_1(\alpha), \ldots, f_n(\alpha)$ satisfy two independent multiplicative relations. This result has been later extended over $\C$ in~\cite{BMZ03} (and in fact over the algebraic closure of any field of characteristic zero), and also relaxed by Maurin \cite{Maurin08} over $\Qbar$ and by Bombieri, Masser and Zannier \cite{BMZ08} over $\mathbb{C}$ showing it holds for rational functions which are multiplicatively independent only.

The analogous problem of linear dependence of points on elliptic curves was considered not long after \cite{BMZ}, for instance in \cite{Viada2003}. 
Let $E$ be an elliptic curve defined over a field $\mathbb K$; we say that $P_1, \ldots, P_n\in E(\overline{\mathbb K})$ are \textit{linearly independent} over a ring $R \subseteq \End(E)$ if there is no non-zero vector $(k_1, \ldots, k_2)\in R^n$ such that 
$$ k_1P_1+\cdots + k_nP_n= O, $$
where we denote by $O$ the point at infinity of the elliptic curve. If this is not the case, then they are called \textit{linearly dependent over $R$}. 

Similarly to the case of rational functions, the points $P_1,\ldots, P_n \in E(\overline{\mathbb K(X)})$ are said to be \textit{linearly independent} over $R \subseteq \End(E)$ \textit{modulo points in} $E(\overline{\mathbb{K}})$ if there is no non-zero vector $(k_1, \ldots, k_n)\in R^n$ such that 
$$ k_1P_1+\cdots + k_nP_n\in E(\overline{\mathbb{K}}).  $$
In this setting, the finiteness result corresponding to Maurin's theorem was proved by Viada in \cite{Viada2008} under some conjecture that was later showed by Galateau in \cite{Gala}.

In this paper, we are interested in studying the multiplicative dependence of elements in the algebraic closure of a finite prime field and the linear dependence of points on elliptic curves modulo primes.

For a prime $p$, let $\ov{\F}_p$ denote the algebraic closure of the field $\F_p$ of $p$ elements. Note that any element of $\overline \F_p^*$ has finite order, so Maurin's finiteness result in characteristic $0$ does not hold in full generality in positive characteristic. In this context, Masser proposed some conjecture in positive characteristic putting more restrictive hypotheses on the rational functions in order to recover Maurin's finiteness result~\cite{Maurin08}, and proved it for $n=3$ \cite[Theorem 1.1]{Masser}.
For other results on unlikely intersections in positive characteristic, see also \cite{GM06,PR13,Scanlon}.

 In this paper we refine the notion of multiplicative dependence over $\ov{\F}_p$ 
defined by \eqref{eq:MultDep}, and we introduce the following concept.

\begin{definition} [\textit{$K$-multiplicative  dependence}] Let $K$ be a positive integer.
We say that elements  $\alpha_1,\ldots,\alpha_n \in \ov \F_p^*$  are
\textit{$K$-multiplicatively  dependent} 
if there exists a non-zero integer vector $(k_1,\ldots,k_n)$ such that    
$$
\alpha_1^{k_1} \cdots \alpha_n^{k_n}=1 \quad \textrm{and} \quad \max_{i=1, \ldots, n} |k_i| \le K.
$$ 
\end{definition}

We  use $\ord_p(\alpha)$ to denote the
multiplicative order of $\alpha \in \ov{\F}_p^*$ (that is, the size of the multiplicative 
group generated by $\alpha$).

Let $E$ be an elliptic curve defined by a Weierstrass equation over the field of rational numbers $\Q$: 
\begin{equation}
\label{eq:EC}
Y^2=X^3+aX+b,\quad a,b \in \Q, \quad 4a^3+27b^2 \ne 0.
\end{equation}

If the reduction of $E$ modulo $p$, denoted by $E_p$, is also an elliptic curve (a sufficient condition for this is that $p$ does not divide the denominators of $a$ and $b$  and that $4a^3+27b^2 \not\equiv 0 \pmod{p}$), 
then for any $\alpha \in \ov{\F}_p$, we define $\ord_{E_p}(\alpha)$ to be the order of the point $(\alpha,\beta)$ 
on the elliptic curve $E_p$ for some $\beta \in \ov{\F}_p$. 
We always denote by $O$ the point at infinity of an elliptic curve. 

\begin{definition} [\textit{$L$-linear  dependence}] Let $L$ be a positive integer.
We say that the  points   $P_1,\ldots,P_n$ on the reduction $E_p$ of the elliptic curve $E$ modulo $p$ (assuming $E_p$ is also an elliptic curve)  are
\textit{$L$-linearly  dependent} 
if there exists a non-zero integer vector $(k_1,\ldots,k_n)$ such that 
\begin{equation}\label{eq:EC-dependence}
k_1P_1 + \cdots + k_nP_n=O \quad \textrm{and} \quad \max_{i=1, \ldots, n} |k_i| \le L.
\end{equation}

Moreover, for any $\alpha_1,\dots, \alpha_n\in \overline{\F}_p$, we say that the points 
\begin{equation}\label{eq:EC-alpha}
(\alpha_1, \cdot),\dots, (\alpha_n,\cdot)
\end{equation}
are \textit{$L$-linearly  dependent} if the points $(\alpha_1, \beta_1),\dots, (\alpha_n,\beta_n)$ are $L$-linearly  dependent for some $\beta_1,\dots, \beta_n \in \overline{\F}_p$ such that $(\alpha_1, \beta_1),\dots, (\alpha_n,\beta_n)\in E_p$.
\end{definition}

We remark that for each $\alpha_i \in\overline{\F}_p$ there is some $\beta_i \in\overline{\F}_p$ such that $(\alpha_i,\beta_i)\in E_p$. 
As the curve $E$ is defined by the Weierstrass equation~\eqref{eq:EC}, the value $\beta_i$ is unique up to sign 
and moreover $-(\alpha_i, \beta_i) = (\alpha_i, -\beta_i)$ on $E_p$. 
Thus, by changing the sign of the coefficients $k_i$ in \eqref{eq:EC-dependence} if necessary, we see that the notion of $L$-linearly dependence of~\eqref{eq:EC-alpha} does not depend on the choices of $\beta_i$. 
Here and there we also say that $(\alpha_i, \cdot)$ is a point (by fixing the second coordinate as $\beta_i$ or $-\beta_i$). 
Moreover, these notions also apply to elliptic curves defined over an arbitrary field.

\section{Main results}

In this section we present the main results of this paper, together with some consequences. The proofs  will be given in Section~\ref{sect:proofs}.

Here and in the rest of the paper, for $\alpha \in \ov{\F}_p$ and $f\in \Q(X)$, the expression $f(\alpha)$ indicates the element of $\ov{\F}_p$ that is obtained by substituting $\alpha$ in the reduction modulo $p$ of the rational function $f$, when this is possible. We implicitly exclude the primes $p$ such that the reductions of the rational functions we are considering are not defined, and such that the reduction of the given elliptic curve modulo $p$ is not an elliptic curve.

Let $E$ be an elliptic curve defined as in \eqref{eq:EC}.  Let $K,L$  be two positive integers, 
and let $\bphi=(\varphi_1,\ldots,\varphi_m)$ and $\brho = (\varrho_1,\ldots,\varrho_n)$ whose components are  all non-zero rational functions in $\Q(X)$. 
Informally, three of our main results can be summarised as follows: under some natural conditions on the involved functions and the curve, 
for any sufficiently large prime $p$ (depending on some parameters, such as $K, L$), one has:  
\begin{itemize}
\item the number of elements $\alpha \in \ov{\F}_p$, for which $\varphi_1(\alpha), \ldots,\varphi_m(\alpha)$ satisfy two independent multiplicative relations with exponents bounded above by $K, L$ in absolute value, respectively, can be upper bounded independently of $p, K, L$;
\item the number of elements $\alpha\in\ov\F_p$, for which $\varphi_1(\alpha),\ldots,\varphi_m(\alpha)$ are $K$-multiplicatively 
dependent and the points 
$$
(\varrho_1(\alpha),\cdot), \ldots,(\varrho_n(\alpha),\cdot)
$$ 
are $L$-linearly dependent on $E_p$, can be upper bounded independently of $p, K, L$; 
\item the number of elements $\alpha \in \ov{\F}_p$, for which 
$(\varrho_1(\alpha),\cdot), \ldots,(\varrho_n(\alpha),\cdot)$ satisfy two independent linear relations over $\Z$ with coefficients bounded above by $K, L$ 
in absolute value, respectively, can be upper bounded independently of $p, K, L$. 
\end{itemize}

In the sequel, we state the above three results precisely and present their consequences. 

It seems that our main results could be generalised  to number fields as well, since all our preliminary results and tools hold over arbitrary number fields (and we shall present some of them in this generality). However, in this paper we present our main results over $\Q$ only. 

Throughout the paper, we will use the Landau symbol $O$ and the Vinogradov symbol $\ll$. Recall that the
assertions $U=O(V)$ and $U \ll V$  are both equivalent to the inequality $|U|\le cV$ with some absolute constant $c>0$.
  To emphasise the dependence of the implied constant $c$ on some parameter (or a list of parameters) $\beta$, 
  we write $U=O_{\beta}(V)$ or $U \ll_{\beta} V$.

\subsection{Multiplicative dependence with two independent relations} 

Given $\bphi = (\varphi_1,\ldots,\varphi_m) \in \Q(X)^m$ a vector of non-zero rational functions, consider the set 
\begin{equation*}
\begin{split}
\mathcal{S}_1 = &\Big\{ \alpha \in \Qbar :  \prod_{i=1}^m\varphi_i(\alpha)^{k_i} =\prod_{i=1}^m\varphi_i(\alpha)^{\ell_i} =1  \text{ for some linearly independent} \\
&\qquad\qquad\qquad\qquad\qquad \qquad\qquad\qquad(k_1,\ldots,k_m),(\ell_1,\ldots,\ell_m)\in\Z^m \Big\}.
\end{split}
\end{equation*}
In defining $\mathcal{S}_1$, we implicitly exclude the poles and zeros of $\varphi_1,\ldots,\varphi_m$. 

As noted in the introduction, by \cite{Maurin08}, if $\varphi_1,\ldots,\varphi_m$ are multiplicatively independent, this set is finite and its cardinality is effectively computable, see Lemma~\ref{lem:eff_Maurin} below.

For positive integers $K,L\geq 1$ and prime $p$, define the set
\begin{equation}
\begin{split}
\label{eq:Set A}
\cA_{\bphi}(p,K,L) =& \Big\{\alpha  \in \ov \F_p: \ \prod_{i=1}^m\varphi_i(\alpha)^{k_i} =\prod_{i=1}^m\varphi_i(\alpha)^{\ell_i} =1  \text{ for some linearly independent}\\
&\qquad (k_1,\ldots,k_m),(\ell_1,\ldots,\ell_m)\in\Z^m,\ \max_{i=1, \ldots, m} |k_i|\le K,\ \max_{i=1, \ldots, m} |\ell_i| \le L\Big\}.
\end {split}
\end{equation}
In defining $\cA_{\bphi}(p,K,L)$, we implicitly assume that the reductions of the rational functions $\varphi_1,\ldots,\varphi_m$ modulo $p$ are all well-defined, and also we implicitly exclude the poles and zeros of the reductions of  $\varphi_1,\ldots,\varphi_m$ modulo $p$. This applies to the other sets in the modulo $p$ setting. 

Our first main result is the following:
\begin{theorem}
\label{thm:Fp-A}
	Let $\bphi = (\varphi_1,\ldots,\varphi_m) \in \Q(X)^m$ whose components are non-zero multiplicatively independent rational functions.
Then,  there exists an effectively computable constant $c_1$ depending only on $\bphi$ 
such that  for arbitrary integers $K, L \ge 1$, 
and any  prime $p > \exp(c_1 KL)$, for the set~\eqref{eq:Set A}
we have 
$$
\#\cA_{\bphi}(p,K,L)  \le \# \mathcal S_1, 
$$
where $\# \mathcal S_1$ is effectively upper bounded, and the elements of $\cA_{\bphi}(p,K,L)$ come from the reduction modulo $p$ of elements of $\mathcal S_1$. 
\end{theorem}

We have the following straightforward consequence of Theorem~\ref{thm:Fp-A}. For this, we define  
\begin{equation}
\begin{split}
\label{eq:Set D}
\cD_{\bphi,\brho}(p,K,L) = \Big \{\alpha  \in \ov \F_p&:~ \text{$\varphi_1(\alpha), \ldots,\varphi_m(\alpha)$ are $K$-multiplicatively dependent} \\
& \text{and $\varrho_1(\alpha), \ldots,\varrho_n(\alpha)$ are $L$-multiplicatively dependent}\Big \}.
\end {split}
\end{equation}

\begin{corollary}
\label{cor:Fp-multdep}
Let $\bphi = (\varphi_1,\ldots,\varphi_m)$ and $\brho = (\varrho_1,\ldots,\varrho_n)$
 whose components are all non-zero rational functions in $\Q(X)$ 
 such that $\varphi_1,\ldots,\varphi_m, \varrho_1,\ldots,\varrho_n$ are multiplicatively independent.
Then,  there are two effectively computable constants $c_1$ and $c_2$, depending only on $\bphi$ and $\brho$,
such that  for arbitrary integers $K, L \ge 1$, 
and any  prime $p > \exp(c_1 KL)$, for the set~\eqref{eq:Set D}
we have 
$$
\#\cD_{\bphi, \brho}(p,K,L)  \le c_2.
$$
\end{corollary}

Taking $m = n =1$ and  $K = L =\rf{c_3(\log p)^{1/2}}$ for some effectively computable constant $c_3$ depending only on $\varphi=\varphi_1$ 
and $\varrho = \varrho_1$ in Corollary~\ref{cor:Fp-multdep}, we directly obtain: 

\begin{corollary}
\label{cor:Ord1} 
Let  $\varphi, \varrho \in \Q(X)$ be non-zero rational functions such that $\varphi, \varrho$ are multiplicatively independent. 
Then,  there are three effectively computable constants $c_1, c_2,c_3$ depending only on $\varphi, \varrho$ 
such that for any prime $p > c_1$, for all but $c_2$ elements $\alpha \in \ov \F_p$ we have 
$$
\max\{\ord_p (\varphi(\alpha)), \ord_p (\varrho(\alpha))\}  \ge c_3 (\log p)^{1/2}. 
$$
\end{corollary}

In Corollary~\ref{cor:Ord1}, we implicitly assume that $\alpha$ is neither a pole nor a zero of $\varphi$ and $\varrho$. 
This applies to similar circumstances in the sequel. 

Corollary~\ref{cor:Ord1}, applied to the curve $Y = \varphi (X)$, improves
a result of Chang~\cite[Theorem 1.1]{Chang2} in this special case which is of the shape 
$$
\max\{\ord_p(\alpha), \ord_p(\varphi(\alpha))\} \gg 
\left(\frac{\log p}{\log \log p}\right)^{1/2}.
$$
The improvement of the same shape has been pointed out in \cite[Section 5]{CKSZ}. 
See \cite{Voloch1} for an earlier work of Voloch.

In view of Corollary~\ref{cor:Ord1}, we further obtain some asymptotic results about the multiplicative orders of those rational values. 

\begin{theorem}
\label{thm:Ord2} 
Let $\bphi=(\varphi_1,\ldots,\varphi_m) \in \Q(X)^m$ be  defined as in Theorem~\ref{thm:Fp-A}. 
Then,  there are two effectively computable constants $c_1, c_2$ depending only on $\bphi$ 
such that  as $N \to \infty$, for all but $c_1 N(\log N)^{-2}$  primes $p \le N$ and for all but at most 
$c_2$ elements $\alpha \in \ov{\F}_p$, at least $m-1$ elements of $\varphi_1(\alpha),\ldots,\varphi_m(\alpha)$ 
are of order at least $(N/\log N)^{1/(2m+2)}$. 
\end{theorem}

We remark that recently Kerr, Mello and Shparlinski \cite[Theorem 2.2]{KMS}, using similar ideas,  
established, for a set of primes $p$ of natural density 1,  a lower bound of the form $p^{1/(2m+2)+o(1)}$ for the order of all but finitely  many vectors $(\varphi_1(\alpha), \ldots, \varphi_m(\alpha))$, 
$\alpha  \in \ov \F_p$, which satisfy two independent multiplicative relations as in the set $\cA_{\bphi}(p,K,L)$. 
One can compare this with Theorem~\ref{thm:Ord2} above.

\begin{theorem}   
\label{thm:Ord-A} 
Let $\bphi = (\varphi_1,\ldots,\varphi_m)$ and $\brho = (\varrho_1,\ldots,\varrho_n)$ be  defined as in Corollary~\ref{cor:Fp-multdep}. 
Then,  there are two effectively computable constants $c_1, c_2$ depending only on $\bphi$ and $\brho$ 
such that  as $N \to \infty$, for all but $c_1 N(\log N)^{-2}$  primes $p \le N$ and for all but at most 
$c_2$ elements $\alpha \in \ov{\F}_p$, at least one 
of the two finitely generated subgroups of $\ov{\F}_p^*$
\[
\langle \varphi_1(\alpha), \ldots,\varphi_m(\alpha) \rangle \quad \mand \quad \langle \varrho_1(\alpha), \ldots,\varrho_n(\alpha) \rangle
\]
is of order at least $N^{mn/(2mn+m+n)}(\log N)^{-1/2}$. 
\end{theorem}

With $m=n=1, \varrho_1 = X$ and  $Y = \varphi_1 (X)$,   
Theorem~\ref{thm:Ord-A} recovers  
the result of Chang~\cite[Theorem~1.2]{Chang2} in this special case (see \cite{CKSZ} for a generalisation to algebraic varieties).

Finally, we obtain a trade-off between the number of possible exceptional 
values $\alpha$ and the parameters $K, L$.

In what follows, $v_p(u)$ denotes the $p$-adic valuation of a non-zero integer $u$ (that is, the highest exponent $v$ such that $p^v$ divides $u$), and we define $v_p(0)=\infty$.

\begin{theorem}
\label{thm:Fp-multdep A} 
Let $\bphi=(\varphi_1,\ldots,\varphi_m) \in \Q(X)^m$ be  defined as in Theorem~\ref{thm:Fp-A}. 
Assume further that $\varphi_1,\ldots,\varphi_m$ are multiplicatively independent modulo constants.  
Then,  there are three effectively computable constants $c_1, c_2, c_3$ depending only on $\bphi$ 
such that   for   arbitrary integers $K,  L \ge 1$, 
there is a positive integer $T$ with 
$$
\log T  \le c_1 (KL)^{m+1}
$$
such that for any prime $p > c_2$ and for the set~\eqref{eq:Set A}
we have 
$$
\#\cA_{\bphi}(p,K,L) \le    v_p(T) + c_3.
$$
\end{theorem}

We remark that in Theorem~\ref{thm:Fp-multdep A} we can not remove the condition ``multiplicatively independent modulo constants". 
For example, let $\varphi_1 = 2, \varphi_2 = 3$. Then for every prime $p$ if we choose $K,L$ to be the multiplicative orders of $2$ and  $3$ modulo $p$ respectively, we have that $\#\cA_{\bphi}(p,K,L)$ is infinite. 

We also point out that, bounding the number of elements of $\cA_{\bphi}(p,K,L)$ by the degrees of the corresponding rational functions, we have 
$$
\#\cA_{\bphi}(p,K,L) \ll_{\bphi} \min\{K^{m+1}, L^{m+1}\}. 
$$
However, in Theorem~\ref{thm:Fp-multdep A} $v_p(T)$ can be zero or bounded by $\log_p(T)$, and it gives 
$$
\#\cA_{\bphi}(p,K,L) \ll_{\bphi} \log_p(T) \ll_{\bphi} \frac{(KL)^{m+1}}{\log p}, 
$$
which gives a better bound when $p$ tends to infinity.

Similarly to Theorem~\ref{thm:Fp-multdep A}, we have: 

\begin{theorem}
\label{thm:Fp-multdep D} 
Let $\bphi = (\varphi_1,\ldots,\varphi_m)$ and $\brho = (\varrho_1,\ldots,\varrho_n)$ be  defined as in Corollary~\ref{cor:Fp-multdep}. 
Assume further that $\varphi_1,\ldots,\varphi_m$ are multiplicatively independent modulo constants. 
Then,  there are three effectively computable constants $c_1, c_2, c_3$ depending only on $\bphi$ and $\brho$ 
such that   for   arbitrary integers $K,  L \ge 1$, 
there is a positive integer $T$ with 
$$
\log T  \le c_1 K^{m+1}L^{n+1}
$$
such that for any prime $p > c_2$ and for the set~\eqref{eq:Set D}
we have 
$$
\#\cD_{\bphi, \brho}(p,K,L) \le    v_p(T) + c_3.
$$
\end{theorem}

\subsection{Multiplicative dependence and linear dependence}

We fix an elliptic curve $E$ defined by \eqref{eq:EC}. Given $\bphi=(\varphi_1, \ldots, \varphi_m) \in \Q(X)^m$ and $\brho=(\varrho_1, \ldots, \varrho_n) \in \Q(X)^n$ two vectors of non-zero rational functions, we define
\begin{equation*}
\begin{split}
 \mathcal{S}_2 =& \big\{ \alpha  \in \overline{\Q}: 
 \varphi_1(\alpha), \ldots,\varphi_m(\alpha) \text{ are multiplicatively}\\
&\qquad\qquad\qquad \text{dependent and $(\varrho_1(\alpha),\cdot), \ldots,(\varrho_n(\alpha),\cdot)$ are linearly dependent}\big\}.
\end {split}
\end{equation*}
In defining $\mathcal{S}_2$, we implicitly exclude the poles and zeros of $\varphi_1,\ldots,\varphi_m$ and the poles of $\varrho_1,\ldots,\varrho_n$.

Under the conditions of Theorem~\ref{thm:EC} below on $\bphi$ and $\brho$, the set $\cS_2$ is finite, as proved in  Lemma~\ref{lem:new2.7}.

For positive integers $K,L\geq 1$ and prime $p$,
define the set 
\begin{equation}
\label{eq:Set B}
\begin{split}
 \cB_{\bphi, \brho, E}(p,K,L) =& \big\{ \alpha  \in \ov \F_p:~ \varphi_1(\alpha), \ldots,\varphi_m(\alpha) \text{ are $K$-multiplicatively}\\
&\quad \text{dependent and $(\varrho_1(\alpha),\cdot), \ldots,(\varrho_n(\alpha),\cdot)$ are $L$-linearly dependent}\big\}.
\end {split}
\end{equation}
In defining $\cB_{\bphi, \brho, E}(p,K,L)$, we implicitly assume that the reductions of the rational functions $\varphi_1,\ldots,\varphi_m, \varrho_1, \ldots, \varrho_n$ modulo $p$ are all well-defined, as well as the reduction $E_p$ is also an elliptic curve.   Moreover, we implicitly exclude the poles and zeros of the reductions of  $\varphi_1,\ldots,\varphi_m$ modulo $p$ and the poles of the reductions of  $\varrho_1,\ldots,\varrho_n$ modulo $p$. This applies to the other sets in the modulo $p$ setting.

For the cardinality $\# \cB_{\bphi, \brho, E}(p,K,L)$, we have:

\begin{theorem}
\label{thm:EC}
Let $E$ be an elliptic curve defined by \eqref{eq:EC}, 
and let $\bphi=(\varphi_1,\dots, \varphi_m)$ and $\brho =(\varrho_1,\dots,\varrho_n)$ whose components are all non-zero rational functions in $\Q(X)$ such that 
$\varphi_1,\dots, \varphi_m$ are multiplicatively independent and 
the points $(\varrho_1(X), \cdot), \ldots, (\varrho_n(X), \cdot)$ in $E(\overline{\Q(X)})$ are linearly independent over $\Z$.  
Suppose moreover that at least one of the following conditions holds:
	\begin{enumerate}
		\item $\varphi_1, \dots , \varphi_m $ are multiplicatively independent modulo constants;
		\item the points $(\varrho_1(X), \cdot), \ldots, (\varrho_n(X), \cdot)$ are linearly independent over the endomorphism ring $\mathrm{End}(E)$ modulo points in $E(\overline{\Q})$.
\end{enumerate}	
Then, there exists an effectively computable constant  $c_1$ depending only on  $\bphi, \brho, E$ such that for any $p>\exp (c_1 K L^2)$, 
for the set \eqref{eq:Set B} we have
$$
\#\cB_{\bphi, \brho, E}(p,K,L)  \leq \# \mathcal S_2,
$$
where the elements of $\cB_{\bphi, \brho, E}(p,K,L)$ come from the reduction modulo $p$ of elements of $\mathcal S_2$, 
and $\# \mathcal S_2$ is effectively upper bounded when $n = 1$. 
\end{theorem}

Taking $m = n =1$ and  $K = L =\rf{c_3(\log p)^{1/3}}$ for some effectively computable constant $c_3$ 
depending only on $E, \varphi=\varphi_1$ and  $\varrho=\varrho_1$  in Theorem~\ref{thm:EC}, we get: 

\begin{corollary}
\label{cor:Ord3} 
Let  $\varphi, \varrho \in \Q(X)$ be non-constant rational functions. 
Then,  there exist three effectively computable constants $c_1, c_2, c_3$  depending only on $\varphi, \varrho, E$ 
such that for any prime $p > c_1$ and for all but $c_2$ elements $\alpha \in \ov \F_p$ we have 
$$
\max\{\ord_p (\varphi(\alpha)), \ord_{E_p} (\varrho(\alpha))\}  \ge c_3 (\log p)^{1/3}. 
$$
\end{corollary}

In Corollary~\ref{cor:Ord3}, we implicitly assume that $\alpha$ is neither a pole nor a zero of $\varphi$  
and also $\alpha$ is not a pole of $\varrho$. 
This applies to similar circumstances in the sequel.

A result of Voloch \cite[Theorem 4.1]{Voloch2} roughly states that for a point $P$ on a fixed elliptic curve over a finite field, 
under some conditions about the order of $P$ and the degree of the field generated by $P$, the order of the $y$-coordinate of $P$ is large. 
So, Corollary~\ref{cor:Ord3} is somehow an extension of Voloch's result in large characteristic.

Moreover, we obtain an asymptotic result compared to Corollary~\ref{cor:Ord3}.

\begin{theorem}
\label{thm:Ord-B} 
Let $E, \bphi=(\varphi_1,\ldots,\varphi_m), \brho=(\varrho_1, \ldots, \varrho_n)$ be  defined as in Theorem~\ref{thm:EC}.  
Then, there exist an effectively computable constant  $c_1$  and a constant $c_2$ 
both depending only on  $\bphi, \brho, E$ 
such that  as $N \to \infty$, for all but $c_1 N(\log N)^{-2}$  primes $p \le N$ and for all but at most 
$c_2$ elements $\alpha \in \ov{\F}_p$,  either  the order of the subgroup 
 $\langle \varphi_1(\alpha), \ldots,\varphi_m(\alpha) \rangle$ in $\ov{\F}_p^*$ or the order of the subgroup 
  $\langle (\varrho_1(\alpha), \cdot), \ldots, (\varrho_n(\alpha), \cdot) \rangle$ in $E(\ov{\F}_p)$ is at least 
$$
N^{mn/(2mn +2m+n)}(\log N)^{-1/2}.
$$ 
Moreover, when $n=1$, the constant $c_2$ is also effectively computable. 
\end{theorem}

Finally, one can also obtain a trade-off between the number of possible exceptional 
values $\alpha$  and the parameters $K, L$. 

\begin{theorem}
\label{thm:EC-multdep B} 
Let $E, \bphi=(\varphi_1,\ldots,\varphi_m), \brho = (\varrho_1, \ldots, \varrho_n)$ be  defined as in Theorem~\ref{thm:EC} 
such that the condition $(1)$ therein is satisfied.  
Then, there exist two effectively computable constants  $c_1, c_2$  and a constant $c_3$ 
all depending only on  $\bphi, \brho, E$ such that  for   arbitrary integers $K,  L \ge 1$, 
there is a positive integer $T$ with 
$$
\log T  \le c_1 K^{m+1}L^{n+2}
$$
such that for any prime $p > c_2$ for which the reduction $E_p$ of $E$ is also an elliptic curve 
 and for the set~\eqref{eq:Set B}
we have 
$$
\#\cB_{\bphi, \brho, E}(p,K,L) \le    v_p(T) + c_3. 
$$
Moreover, when $n=1$, the constant $c_3$ is also effectively computable. 
\end{theorem}

\subsection{Linear dependence with two independent relations}    
Given $\brho=(\varrho_1, \ldots, \varrho_n) \in \mathbb Q(X)^n$ a vector of non-zero rational functions, define the set
\begin{equation*}
\begin{split}
\cS_3= & \Big\{\alpha  \in \ov \Q: \ \sum_{i=1}^n k_i(\varrho_i(\alpha),\cdot)=\sum_{i=1}^n\ell_i(\varrho_i(\alpha),\cdot)=O   \text{ for some linearly}\\
&\qquad\qquad\qquad\qquad\qquad\qquad \text{ independent }\  (k_1,\ldots,k_n),(\ell_1,\ldots,\ell_n)\in\Z^n \Big\}.
\end {split}
\end{equation*}
In defining $\mathcal{S}_3$, we implicitly  exclude the poles of $\varrho_1,\ldots,\varrho_n$.

As mentioned in the introduction, under the conditions of Theorem~\ref{thm:EC-C} below on $\brho$, the set $\cS_3$ was proved to be finite by Viada~\cite{Viada2008} and Galateau~\cite{Gala}, see Lemma~\ref{lem:Gala}.

For positive integers $K,L\geq 1$ and prime $p$,
define the set:
\begin{equation}
\label{eq:Set C}
\begin{split}
\mathcal{C}_{\brho,E}&(p,K,L) = \Big\{\alpha  \in \ov \F_p: \ \sum_{i=1}^n k_i(\varrho_i(\alpha),\cdot)=\sum_{i=1}^n\ell_i(\varrho_i(\alpha),\cdot)=O   \text{ for some linearly}\\
&\quad \text{ independent }\  (k_1,\ldots,k_n),(\ell_1,\ldots,\ell_n)\in\Z^n,\ \max_{i=1, \ldots, n} |k_i|\le K,\ \max_{i=1, \ldots, n} |\ell_i| \le L\Big\}.
\end {split}
\end{equation}

For the cardinality $\# \cC_{\brho, E}(p,K,L)$, we have:

\begin{theorem}
	\label{thm:EC-C}
	Let $\brho = (\varrho_1,\ldots,\varrho_n) \in \Q(X)^n$ be a vector of non-zero rational functions such that the points $(\varrho_1(X), \cdot), \ldots, (\varrho_n(X), \cdot)$ in $E(\overline{\Q(X)})$ are linearly independent over $\mathrm{End}(E)$.
Then, there exists an effectively computable constant  $c_1$ depending only on $\brho, E$ 	such that  for arbitrary integers $K, L \ge 1$, 
and any  prime $p > \exp(c_1 K^2L^2)$, for the set~\eqref{eq:Set C}	we have 
	$$
	\#\mathcal{C}_{\brho, E}(p,K,L)  \le \# \mathcal S_3,
	$$
	and the elements of $\mathcal{C}_{\brho, E}(p,K,L)$ come from the reduction modulo $p$ of elements of $\mathcal S_3$.
\end{theorem}

We have the following straightforward consequence about the set: 
\begin{equation}
\label{eq:Set E}
\begin{split}
& \cE_{\bphi, \brho, E}(p,K,L) = \big\{ \alpha  \in \ov \F_p:~ (\varphi_1(\alpha),\cdot), \ldots,(\varphi_m(\alpha), \cdot) \text{ are $K$-linearly}\\
&\qquad\qquad \text{dependent and $(\varrho_1(\alpha),\cdot), \ldots,(\varrho_n(\alpha),\cdot)$ are $L$-linearly dependent}\big\}.
\end {split}
\end{equation}

\begin{corollary}
\label{cor:EC-lidep}
Let $\bphi = (\varphi_1,\ldots,\varphi_m)$ and $\brho = (\varrho_1,\ldots,\varrho_n)$
 whose components are all non-zero rational functions in $\Q(X)$ such that the points 
$(\varphi_1(X), \cdot), \ldots, (\varphi_m(X), \cdot)$, $(\varrho_1(X), \cdot), \ldots, (\varrho_n(X), \cdot)$ in $E(\overline{\Q(X)})$ are linearly independent over $\mathrm{End}(E)$. Then, there exist an effectively computable constant  $c_1$  and a constant $c_2$ 
	both depending only on $\bphi, \brho, E$ 	such that  for arbitrary integers $K, L \ge 1$, 
and any  prime $p > \exp(c_1 K^2L^2)$, for the set~\eqref{eq:Set D}
we have 
$$
\#\cE_{\bphi, \brho,E}(p,K,L)  \le c_2.
$$
\end{corollary}

Taking $m = n =1$ and  $K = L =\rf{c_3(\log p)^{1/4}}$ for some effectively computable constant $c_3$ depending only on 
$E, \varphi=\varphi_1$ and $\varrho = \varrho_1$ in Corollary~\ref{cor:EC-lidep}, we directly have: 

\begin{corollary}
\label{cor:EC-Ord} 
Let  $\varphi, \varrho \in \Q(X)$ be non-zero rational functions such that the two points
$(\varphi(X), \cdot),  (\varrho(X), \cdot)$ in $E(\overline{\Q(X)})$ are linearly independent over $\mathrm{End}(E)$. 
Then, there exist two effectively computable constants  $c_1, c_3$  and a constant $c_2$ 
all depending only on $\varphi, \varrho, E$ 
such that for any prime $p > c_1$, for all but $c_2$ elements $\alpha \in \ov \F_p$ we have 
$$
\max\{\ord_{E_p} (\varphi(\alpha)), \ord_{E_p} (\varrho(\alpha))\}  \ge c_3 (\log p)^{1/4}. 
$$
\end{corollary}

In the elliptic case we prove similar results as in Theorems~\ref{thm:Ord2}, ~\ref{thm:Ord-A} and~\ref{thm:Ord-B}. 

\begin{theorem}
\label{thm:EC-Ord2} 
Let $\brho = (\varrho_1,\ldots,\varrho_n) \in \Q(X)^n$ be  defined as in Theorem~\ref{thm:EC-C}. 
Then, there exist an effectively computable constant  $c_1$  and a constant $c_2$ 
	both depending only on $\brho, E$ 	
such that  as $N \to \infty$, for all but $c_1 N(\log N)^{-2}$  primes $p \le N$ and for all but at most 
$c_2$ elements $\alpha \in \ov{\F}_p$, at least $n-1$ points of $(\varrho_1(\alpha),\cdot),\ldots,(\varrho_n(\alpha),\cdot)$ 
are of order at least $(N/\log N)^{1/(2n+4)}$. 
\end{theorem}

\begin{theorem}   
\label{thm:EC-Ord3} 
Let $\bphi = (\varphi_1,\ldots,\varphi_m)$ and $\brho = (\varrho_1,\ldots,\varrho_n)$ be  defined as in Corollary~\ref{cor:EC-lidep}. 
Then, there exist an effectively computable constant  $c_1$  and a constant $c_2$ both depending only on $\bphi, \brho, E$ 
such that  as $N \to \infty$, for all but $c_1 N(\log N)^{-2}$  primes $p \le N$ and for all but at most 
$c_2$ elements $\alpha \in \ov{\F}_p$, at least one 
of the two finitely generated groups 
$$
\langle (\varphi_1(\alpha), \cdot), \ldots,(\varphi_m(\alpha),\cdot) \rangle \quad \mand \quad 
\langle (\varrho_1(\alpha),\cdot), \ldots,(\varrho_n(\alpha),\cdot) \rangle
$$
is of order at least $N^{mn/(2mn+2m+2n)}(\log N)^{-1/2}$.  
\end{theorem}

We end this section with a remark on obtaining an analogue result as in Theorems~\ref{thm:Fp-multdep A},~\ref{thm:Fp-multdep D} and~\ref{thm:EC-multdep B} for the set $\mathcal{C}_{\brho, E}(p,K,L)$. We point out that, in order to prove Theorems~\ref{thm:Fp-multdep A},~\ref{thm:Fp-multdep D} and~\ref{thm:EC-multdep B} we need to show that certain polynomials are not constant modulo $p$, for primes $p$ large enough \emph{independently of $K$ and $L$} (see Section~\ref{sect:proofs} for more details). While this is easily done in the $\G_m$ case, in the setting of Theorem~\ref{thm:EC-C} it is not clear how to get the same uniform bound on $p$. For this reason we do not obtain the analogue of Theorems~\ref{thm:Fp-multdep A},~\ref{thm:Fp-multdep D} and~\ref{thm:EC-multdep B} for the set $\mathcal{C}_{\brho, E}(p,K,L)$.

\section{Preliminaries}

\subsection{Heights of polynomials and rational functions} 

For any non-zero polynomial $f \in \C[X]$, we define the height of $f$, denoted by $H(f)$, 
to be the maximum of the absolute values of its coefficients, and we also define 
$$
\h(f) = \max\{0, \log H(f)\}. 
$$  
If $f(X) = a_d \prod_{i=1}^{d}(X-\alpha_i)$ with $a_d \ne 0$, then the Mahler measure of $f$ is defined to be 
$$
M(f) = |a_d| \prod_{i=1}^{d} \max\{|\alpha_i|, 1\}. 
$$
It is well-known that (see, for instance, \cite[Equation (3.12)]{Wal})
\begin{equation}  \label{eq:Mahler}
2^{-d}H(f) \le M(f) \le \sqrt{d+1}H(f). 
\end{equation}

The following bound on the height of a product of several polynomials is well-known
and holds  in much broader
generality; see, for example,~\cite[Lemma~1.2 (1.b)]{KPS}. 

\begin{lemma}
\label{lem:HeightPoly}
Let $f_1, \ldots, f_n \in \C[X]$ be non-zero polynomials.  Then   
$$
 \h\(\prod_{i=1}^{n} f_{i}\) \le \sum_{i=1}^{n} \( \h(f_i) + \deg f_{i}\). 
$$
\end{lemma}

Clearly, the notion of height naturally extends to multivariate polynomials, namely for non-zero polynomial $f\in \C[X_1,\dots, X_n]$ we let $H(f)$ be the maximum of the absolute values of its coefficients and $\h(f)=\max\{0,\log H(f)\}$.

Moreover, for a rational function $R=f/g$, where $f,g\in \Z[X_1,\dots, X_n]$ are coprime, we define  $\deg R = \max\{\deg f, \deg g\}$, 
$$
H(R)=\max\{H(f),H(g)\} \mand \h(R)=\max\{\h(f),\h(g)\}.
$$

We need the following estimate on the height of composition of rational functions, which is a special case of~\cite[Lemma 3.3]{DOSS}.

\begin{lemma}\label{lem:hight-composition}
  Let $R \in \Q[X_1,\ldots,X_n]$ and $f_1,\ldots, f_n \in \Q(X)$. Set $
  d=\max_{i=1,\ldots, n}\deg f_i$ and $ h=\max_{i=1,\ldots, n}\h(f_i)$. Then
\begin{align*}
\deg R(f_1,\ldots, f_n)& \le dn \deg R, \\
 \h(R(f_1,\ldots, f_n))&\leq \h(R) + h \deg R+  (3dn+1)\log(n+1) \deg R .
\end{align*}
\end{lemma}

\subsection{The size and divisibility of resultants} 
We start with the following simple estimate on the absolute value of the resultant of two polynomials; see \cite[Theorem 6.23]{GG}. 
Its proof relies on applying Hadamard's inequality to the Sylvester matrix of $f$ and $g$. 

\begin{lemma}\label{lem:res}  
 Let $f,g \in \C[X]$ be non-zero polynomials. Then, their resultant $\Res(f,g)$ satisfies 
 $$
 |\Res(f,g)| \leq \left( \sqrt{\deg f+1}\, H(f) \right)^{\deg g} \left( \sqrt{\deg g+1}\, H(g) \right)^{\deg f}.
 $$
\end{lemma}

Now, given two non-zero polynomials $f,g \in \Z[X]$ which are not both constants, 
it is well-known that, if their reductions modulo a prime $p$ have a common factor, 
 then their resultant $\Res(f,g)$ is divisible by
$p$.  The following result of  G\'omez-P\'erez,
Guti\'errez, Ibeas and Sevilla~\cite{GGIS} refines this property for polynomials
with several common roots modulo $p$. 
We remark that in \cite{GGIS} the authors assume the two polynomials to be both non-zero modulo $p$, but in the proof they only need one of the two polynomials not to vanish modulo $p$.  

\begin{lemma}
 \label{lem:Res}
 Let $f,g\in \Z[X]$ be two non-zero polynomials whose reductions modulo $p$ do not both vanish  
   and have  $m$ common roots in $\ov \F_p$, counted
  with multiplicities.  Then, 
\[
m \le v_p ( \Res(f,g) ) .
\]
\end{lemma}

\subsection{Division polynomials and their heights}  \label{sec:div-poly}
Let $E$ be an elliptic curve defined  as in~\eqref{eq:EC}. For any integer $n \ge 1$, let $\psi_n$ be the $n$-th division polynomial of $E$;
 see~\cite[Exercise 3.7]{Silv}  for their definition and properties. 
That is, 
\begin{equation*}  \label{eq:divpoly}
\begin{split}
& \psi_0 = 0, \qquad \psi_1 = 1, \qquad \psi_2 = 2Y, \\
& \psi_3 = 3X^4 +6aX^2 +12bX - a^2, \\
& \psi_4 = 4Y(X^6 +5aX^4 +20bX^3 - 5a^2 X^2 - 4abX - 8b^2 - a^3), \\
& \psi_{2m+1} = \psi_{m+2}\psi_m^3 - \psi_{m-1}\psi_{m+1}^3  \quad \textrm{for $m \ge  2$}, \\
& \psi_{2m} = (2Y)^{-1}\psi_m(\psi_{m+2}\psi_{m-1}^2 - \psi_{m-2}\psi_{m+1}^2 )   \quad \textrm{for $m \ge  3$}.
\end{split}
\end{equation*} 
We remark that the polynomials $\psi_n$ are reduced by the curve equation~\eqref{eq:EC}, 
in particular $\psi_{n}\in \Z[a,b,X]$ if $n$ is odd, $\psi_{n}\in Y\Z[a,b,X]$ if $n$ is even, and $\psi_n^2 \in \Z[a,b,X]$ for any $n \ge 0$. 
Moreover (see, for instance, \cite[Lemma 3.3]{Was} and the proof of \cite[Lemma 3.5]{Was}), we have 
\begin{equation}\label{eq:div-degree}
\begin{split}
& \psi_n = nX^{(n^2-1)/2} + \textrm{(lower degree terms)} \in \Z[a,b,X] \quad \textrm{for odd $n$}, \\
& \psi_n/Y = nX^{(n^2-4)/2} + \textrm{(lower degree terms)} \in \Z[a,b,X] \quad \textrm{for even $n$}. 
\end{split}
\end{equation}
For any integer $n \ge 1$, let $\Psi_n \in \Z[a,b,X]$ be defined by 
\begin{equation}
\label{eq:Psi}
\Psi_n = \left\{\begin{array}{ll}
 \psi_n &\quad\text{if $n$ is odd},\\
 \psi_n / Y &\quad\text{if $n$ is even}.
\end{array}\right.
\end{equation}
By convention, put $\Psi_0 = 0$. 

 We also define
$$
\phi_n=X \psi_n^2-\psi_{n+1}\psi_{n-1}  \qquad n \ge 1, 
$$
where as before, $\phi_n$ is reduced by the curve equation~\eqref{eq:EC}, in particular, $\phi_n\in \Z[a,b,X]$. 
By convention, put $\phi_0 = 0$. 

We note that an affine point $P=(x,y)$ on $E$ is $n$-torsion if $\Psi_n(x) = 0$ for $n \ge 3$ and 2-torsion if $y=0$. 
On the other hand, if $P=(x,y)$ is not an $n$-torsion point, then by~\cite[Theorem 3.6]{Was}, 
the first coordinate of the point $nP$ is 
\begin{equation}   \label{eq:nP}
 \frac{\phi_n (x)}{\psi_n^2(x)}.
\end{equation}

The following lemma follows directly from~\cite[Corollary 1]{McK} 
(note that the polynomials $\Psi_n$ here are exactly the division polynomials $f_n$ defined in \cite{McK}). 

\begin{lemma}
\label{lem:Psi-height}
There exists an effectively computable constant $c$ depending only on $E$ such that, for any integer $n\geq 1$, we have
$$
\h(\Psi_n) \le cn^2. 
$$
\end{lemma}

We conclude this section by giving a bound on the height of the polynomials $\phi_n$.

\begin{lemma}\label{lem:division}
 There exists an effectively computable constant $c$ depending only on $E$ such that, for any integer $n\geq 1$, we have
 $$
  \h(\phi_n)\leq c n^{2}.
 $$
\end{lemma}

\begin{proof}
If $n$ is odd, then we have 
$$
\phi_n = X\Psi_n^2 - (X^3 + aX + b)\Psi_{n+1}\Psi_{n-1}, 
$$
and so, using Lemma~\ref{lem:HeightPoly} and noticing $H(\Psi_n) \ge n$ by \eqref{eq:div-degree} we obtain 
\begin{align*}
\h(\phi_n) & \le \h(\Psi_n^2) + \h((X^3+aX+b)\Psi_{n+1}\Psi_{n-1})\\ 
&\leq 2 (\h(\Psi_n)+\deg \Psi_n) + \h(\Psi_{n+1}) + \h(\Psi_{n-1})\\
& \quad +\h(X^3+aX+b)+ \deg \Psi_{n+1} + \deg \Psi_{n-1}+3. 
\end{align*}
In addition, if $n$ is even, then since 
$$
\phi_n = X (X^3 + aX + b)\Psi_n^2 - \Psi_{n+1}\Psi_{n-1}, 
$$
as the above we obtain 
\begin{align*}
\h(\phi_n) & \le \h((X^3+aX+b)\Psi_n^2) + \h(\Psi_{n+1}\Psi_{n-1}) \\ 
& \leq 2 (\h(\Psi_n)+\deg \Psi_n) +\h(X^3+aX+b)+3 \\
& \quad + \h(\Psi_{n+1}) + \h(\Psi_{n-1}) + \deg \Psi_{n+1} + \deg \Psi_{n-1}.
\end{align*}
Now, the desired result follows from \eqref{eq:div-degree} and Lemma~\ref{lem:Psi-height}.
\end{proof}

\subsection{Summation polynomials}
In this section we recall \textit{summation polynomials} of elliptic curves introduced by Semaev \cite{Semaev}, and bound the height of such polynomials. 

\begin{lemma}\label{lem:summation-polys}
Let $E$ be an elliptic curve of the form~\eqref{eq:EC}  defined over a field $\K$ of characteristic different from $2$ and $3$. 
For any integer $n \geq 2$, there exists a polynomial $\sigma_n \in \Z[X_1,\dots, X_n,a,b]$ 
$($called the $n$-th summation polynomial$)$ with the following property: 
for any $x_1,\dots, x_n \in \overline{\K}$, we have $\sigma_n(x_1,\dots, x_n)=0$ if and only if there are $y_1,\dots, y_n \in \overline{\K}$ 
such that $(x_i,y_i)\in E, 1\leq i\leq n$, and $(x_1,y_1)+\dots+(x_n,y_n)=O$ on the curve. Moreover, the polynomials $\sigma_n$ can be defined by 
\begin{equation} \label{eq:summation-pol}
\begin{split}
\sigma_2(X_1,X_2) & =X_1-X_2, \\ 
\sigma_3(X_1,X_2,X_3) & =(X_1-X_2)^2X_3^2 - 2\left((X_1+X_2)(X_1X_2+a)+2b \right) X_3 \\
            & \qquad +(X_1X_2-a)^2-4b(X_1+X_2), \\
 \sigma_n(X_1,\dots, X_n) & =  \Res_X\left(\sigma_{n-k}(X_1,\dots, X_{n-k-1},X), \sigma_{k+2}(X_{n-k}, \dots, X_n,X) \right)
\end{split}
\end{equation}
for any $n \geq 4$ and $1\leq k \leq n-3$, where $\Res_X$ denotes the resultant with respect to the variable $X$. 
 
For any $n \ge 3$, $\sigma_n$ is an irreducible symmetric polynomial which has degree $2^{n-2}$ in each variable.
\end{lemma}

We now bound the height of summation polynomials. 

\begin{lemma}\label{lem:height_of_summation}   
Let $E$ and $\sigma_n \in \Z[X_1,\dots, X_n,a,b]$, $n \geq 2$, be  defined as in Lemma~\ref{lem:summation-polys}, 
$\K = \C$ and $\sigma_n$ of the form \eqref{eq:summation-pol}.
Then
$$
\h(\sigma_n)=\exp\big( O(n) \big),
$$
where the implied constant is effectively computable depending only on the curve $E$.
\end{lemma}
\begin{proof}
We proceed by induction. Assume that $n \ge 4$ and 
\begin{equation}  \label{eq:k<n}
\h(\sigma_j)\leq \exp(cj), \quad 2\leq j < n, 
\end{equation}
for some constant $c$. Write $\sigma_n$ in the form~\eqref{eq:summation-pol} with $k = \lfloor (n-1)/2 \rfloor$.

Put
\begin{align*}
d &= \deg_X \sigma_{n-k}(X_1,\dots, X_{n-k-1},X),  \\
m &= \deg_X \sigma_{k+2}(X_{n-k}, \dots, X_n,X).
\end{align*} 
By definition, $\sigma_n$ is the determinant of the Sylvester matrix of the polynomials 
$\sigma_{n-k}(X_1,$ $\dots, X_{n-k-1},X)$ and $\sigma_{k+2}(X_{n-k}, \dots, X_n,X)$ with respect to $X$. 
By expanding this determinant, we know that $\sigma_n$ is the sum of at most 
(because by Lemma~\ref{lem:summation-polys}, $d = 2^{n-k-2}$ and $m = 2^k$)
\begin{equation}\label{eq:number_of_summands}
(d+1)^m (m+1)^d = \exp( \exp( O(n)))
\end{equation}
summands of the form
\begin{equation}\label{eq:prod}
f_1 \dots f_m g_1 \dots g_d, 
\end{equation}
where $f_i\in \Z[X_1, ..., X_{n-k-1}]$ and $g_j\in\Z[X_{n-k}, ..., X_n]$ are coefficients of  $\sigma_{n-k}$ and $\sigma_{k+2}$ respectively 
considered as polynomials with respect to the variable $X$. Clearly, for each $i$ and each $j$, 
\begin{equation}\label{eq:bound_on_hight}
\h(f_i)\leq \h(\sigma_{n-k})  \mand \h(g_j)\leq \h(\sigma_{k+2}). 
\end{equation}
In addition, by Lemma~\ref{lem:summation-polys} the degree of $f_i$ and $g_j$ in each variable is at most $2^{n-k-2}$ and $2^{k}$, respectively, 
thus $f_i$ and $g_j$ have at most $(2^{n-k-2} + 1)^{n-k-1}$ and $(2^k + 1)^{k+1}$ nonzero terms. Then, expanding all the products in \eqref{eq:prod}, the maximal number of common monomials is at most 
\begin{equation}\label{eq:monomials}   
(2^{n-k-2} + 1)^{m(n-k-1)} \cdot (2^k + 1)^{d(k+1)} = \exp(\exp( O(n))).
\end{equation}
Hence, it follows from~\eqref{eq:k<n}, \eqref{eq:number_of_summands}, \eqref{eq:prod}, \eqref{eq:bound_on_hight} and \eqref{eq:monomials} that
\begin{align*}
\h(\sigma_n) &\leq \exp(O(n)) + \exp(O(n)) + m \cdot \h(\sigma_{n-k})+ d \cdot \h(\sigma_{k+2})\\
&\leq \exp(O(n)) + 2^{(n-1)/2} \exp(c(n-k)) + 2^{(n-1)/2}\exp(c(k+2)) \\
&\leq \exp(O(n)) + 2^{(n+1)/2} \exp(c(n/2+3/2)),
\end{align*}
and the desired result follows by choosing the constant $c$ large enough. 
Moreover, since the implied constants in both \eqref{eq:number_of_summands} and \eqref{eq:monomials} are effectively computable, 
the constant $c$ is also effectively computable. 
\end{proof}

\subsection{Unlikely Intersections results} 

In this section, we list a series of results about multiplicative dependence of rational functions in $\Qbar(X)$ and linear dependence on elliptic curves for points defined over $\Qbar$. Namely, in order to prove Theorems~\ref{thm:Fp-A}, \ref{thm:EC} and \ref{thm:EC-C}, one needs first to consider the analogous problems in $\Qbar$ (rather than $\overline{\mathbb{F}}_p$), and to show finiteness results. These problems fit in the more general framework of problems of unlikely intersections, which have been deeply studied in the last decades (see, for instance, \cite{Zannier}). 

More specifically, the following lemma is an effective version of a result of Maurin \cite[Th{\'e}or{\`e}me 1.2]{Maurin08} concerning multiplicative dependence of values of rational functions in $\Q(X)$, which in fact was initially proved by Bombieri, Masser and Zannier~\cite{BMZ} under a more restrictive condition of multiplicative independence of the involved functions modulo constants.

\begin{lemma}
\label{lem:eff_Maurin}
Let $\K$ be a number field and let $f_1,\ldots, f_m \in \K(X)$ be non-zero multiplicatively independent rational functions defined over $\K$. Then, the cardinality of the set of $\alpha \in \Qbar$ for which there exist linearly independent vectors $(k_1, \ldots, k_m)$ and $(\ell_1, \ldots, \ell_m) \in \Z^m$ such that 
\[ f_1(\alpha)^{k_1}\cdots f_m(\alpha)^{k_m}=f_1(\alpha)^{\ell_1}\cdots f_m(\alpha)^{\ell_m}=1 \]
is bounded by an effectively computable constant $C$ depending only on $\K$ and $f_1,\ldots, f_m$.
\end{lemma}

\begin{proof}
The ineffective version of this result was proved by Maurin in~\cite[Th{\'e}or{\`e}me 1.2]{Maurin08}. In~\cite{BHMZ10}, Bombieri, Habegger, Masser and Zannier gave a different argument to prove~\cite[Th{\'e}or{\`e}me 1.2]{Maurin08}, showing that effectivity would follow from an effective version of Habegger's theorem \cite{Hab09}. This was finally proved by Habegger himself in \cite{Hab17}.
\end{proof}

The following lemma is a special case of  \cite[Th{\' e}or{\` e}me H]{Gala}. The latter was a conditional result due to Viada \cite{Viada2008} and made unconditional by Galateau \cite{Gala}.

\begin{lemma}  \label{lem:Gala}
	Let $E$ be an elliptic curve defined over a number field $\K$ by  a Weierstrass equation, and let $n \ge 1$ be an integer. 
	Let $\cC$ be an irreducible curve in $E^n$, also defined over $\K$, with coordinates $(X_1,Y_1,\dots ,X_n,Y_n )$ such that the points $(X_j,Y_j)$ are linearly independent over $\mathrm{End}(E)$.
	Then, there are at most finitely many  points $c\in \cC(\C)$ such that $(X_j(c),Y_j(c))$, $j=1,\ldots,n$, 
	satisfy two independent linear relations over $\mathrm{End}(E)$. 
\end{lemma}

The following result is a special case of Corollary~\ref{cor:UI} in Section~\ref{sect:UI}. 
It will be used  in the proof of Theorem~\ref{thm:EC}. 

\begin{lemma}\label{lem:new2.7}
	Let $E$ be an elliptic curve defined over a number field $\K$ by a Weierstrass equation, and let $\bphi=(\varphi_1,\dots, \varphi_m)$ and $\brho =(\varrho_1,\dots,\varrho_n)$ be vectors of non-zero rational functions in $\K(X)$ such that 
	$\varphi_1,\dots, \varphi_m$ are multiplicatively independent and 
	the points $(\varrho_1(X), \cdot), \ldots, (\varrho_n(X), \cdot)$ in $E(\overline{\Q(X)})$ are linearly independent over $\Z$. 
	Suppose moreover that at least one of the following conditions holds:
	\begin{enumerate}
		\item $\varphi_1, \dots , \varphi_m $ are multiplicatively independent modulo constants;
		\item the points $(\varrho_1(X), \cdot), \ldots, (\varrho_n(X), \cdot)$ are linearly independent over $\mathrm{End}(E)$ modulo points in $E(\overline{\Q})$.
	\end{enumerate}	
Then, there are at most finitely many $\alpha \in \overline{\Q}$ such that $\varphi_1(\alpha), \ldots, \varphi_m(\alpha)$ are multiplicatively dependent and the points $(\varrho_1(\alpha), \cdot), \ldots, (\varrho_n(\alpha), \cdot)$ in $E(\overline{\Q})$ are linearly dependent over~$\Z$.
\end{lemma}

In full generality, the proof of this result is in principle not effective, and this makes the constant $c_2$ in the statement of Theorem \ref{thm:EC} ineffective. If $n=1$, the condition of linear dependence of the point $(\varrho_1(\alpha), \cdot)$ means that it is a torsion point, and in this case, it is possible to give an effective version of Lemma~\ref{lem:new2.7} when $\K = \Q$, which is the content of the following lemma.
   
\begin{lemma}   \label{lem:tors}
In Lemma~\ref{lem:new2.7} when $n=1$ and $\K=\Q$, the order of the torsion point $(\varrho_1(\alpha), \cdot)$ can be effectively upper bounded uniformly, and in particular, the cardinality of the set of $\alpha \in \Qbar$ such that $\varphi_1(\alpha), \ldots, \varphi_m(\alpha)$ are multiplicatively dependent and $(\varrho_1(\alpha), \cdot)$ has finite order can be effectively bounded. 
\end{lemma}

\begin{proof}
By \cite[Theorem 1.4]{BS}, there exists an effectively computable bound $B$ for the order of $(\varrho_1(\alpha), \cdot)$ depending on $E,\bphi, \rho_1$ and a lower bound $\epsilon$ for the height of elements of $\Q(E_\tor)^*\setminus \mu_\infty$. Such an effective lower bound is provided by Frey \cite[Theorem 1.2]{Frey17} and, in case $E$ has complex multiplication, by Amoroso and Zannier \cite{AZ10}.
\end{proof}

\section{Unlikely intersections in $ \Gm^m\times E^n$} \label{sect:UI}

To prove Theorem \ref{thm:EC}, we will need a finiteness result for multiplicative relations for rational functions in $\Qbar(X)$ and linear dependence on elliptic curves.
This will follow from  a general statement about the intersection of an irreducible curve in the split semiabelian variety $\Gm^m \times E^n$ with the algebraic subgroups of codimension at least $2$, which we prove in this section. This result fits in the more general framework of unlikely intersections, and it is a particular case of the well known Zilber-Pink conjecture (for an account on these problems, see \cite{Zannier}). 

\begin{theorem}\label{thm:UI}
	Let $E$ be an elliptic curve defined over a number field $\K$ by  a Weierstrass equation, and let $m, n \ge 1$ be integers. 
	Let $\cC$ be an irreducible curve in $ \Gm^m\times E^n$, also defined over $\K$, with coordinates $(Z_1, \dots , Z_m,X_1,Y_1,\dots ,X_n,Y_n )$ such that $Z_1, \dots , Z_m $ are multiplicatively independent and the points $(X_i,Y_i)$ are linearly independent over $\mathrm{End}(E)$. Suppose moreover that at least one of the following conditions holds:
	\begin{enumerate}
		\item $Z_1, \dots , Z_m $ are multiplicatively independent modulo constants;
		\item the points $(X_i,Y_i)$ are linearly independent over $\mathrm{End}(E)$ modulo points in $E(\overline{\Q})$.
	\end{enumerate}	
	Then, there are at most finitely many  points $c\in \cC(\C)$ such that $Z_i(c)$, $i=1,\ldots,m$, are multiplicatively dependent and $(X_j(c),Y_j(c))$, $j=1,\ldots,n$, are linearly dependent over $\mathrm{End}(E)$. 
\end{theorem}

We point out that the Zilber-Pink conjecture for a curve in $ \Gm^m\times E^n$ predicts that the same conclusion of Theorem \ref{thm:UI} should hold without assuming the condition (1) or (2). In a work in progress \cite{BarKS}, the first author, K\"uhne and Schmidt prove the Zilber-Pink conjecture for a curve in a semiabelian variety over the algebraic numbers. As a special case, this would imply Theorem \ref{thm:UI}, and therefore Theorem \ref{thm:EC}, without these unnecessary hypotheses.
\medskip

The proof of Theorem \ref{thm:UI} can be obtained by adapting the proof of \cite[Theorem 1.2]{BC} to this setting. In particular, one follows the general strategy introduced by Pila and Zannier in \cite{PilaZannier} using the theory of o-minimal structures to give an alternative proof of the Manin-Mumford conjecture for abelian varieties. The strategy is based on the combination of various results coming from o-minimality, Diophantine geometry and transcendence results.

An important ingredient of the proof is the well-known Pila-Wilkie Theorem \cite{PilaWilkie} which provides an estimate for the number of rational points on a ``sufficiently transcendental'' real subanalytic variety. Using abelian logarithms, these rational points correspond to torsion points. For more details about the general strategy and how it has been applied to other problems we refer to \cite{Zannier}.

On the other hand, if one wants to deal with points lying in proper algebraic subgroups like in Theorem \ref{thm:UI}, a more refined result is needed. For instance, first in \cite{BC2016} and then in \cite{BC}, the authors adapted ideas introduced in \cite{CMPZ} to deal with linear relations rather than just with torsion points.

\medskip

Let $\mathfrak{h}$  denote the absolute logarithmic Weil heights on $\Gm(\Qbar)$ and on $E(\Qbar)$, 
and let us define a height $\tilde{\mathfrak{h}}$ on $\cC(\Qbar)$ by
$$
\tilde{\mathfrak{h}}(c):=\mathfrak{h}(Z_1(c))+\dots + \mathfrak{h}(Z_m(c))+\mathfrak{h}(X_1(c),Y_1(c))+\dots + \mathfrak{h}(X_n(c),Y_n(c)).
$$

We call $\cC_0$ the set of such points of $\cC(\C)$ that we want to prove to be finite in Theorem \ref{thm:UI}.
First, we note that the points in $\cC_0$ must be algebraic. Moreover, as at least one of the conditions (1) and (2) in Theorem~\ref{thm:UI} holds, $\cC_0$ is a set of bounded height respectively by  
\begin{enumerate}   
	\item \cite[Theorem 1]{BMZ};
	\item \cite[Theorem 1]{Viada2003}.
\end{enumerate}
Indeed, if $V_1$ and $V_2$ are any two coordinates of $\cC$, then there exists a polynomial $f\in \K[T_1,T_2]\setminus \{0\}$ such that $f(V_1,V_2)=0$. Suppose $V_1$ is non-constant and $\mathfrak{h}(V_1(c))\leq B$ for some $c\in \cC(\Qbar)$ and $B\geq 0$; then, because $f(V_1(c),V_2(c))=0$, we get $\mathfrak{h}(V_2(c))\ll B$ and the implied constant depends only on $f$.

We then just have to exhibit a bound on their degree over the number field $\K$.

\begin{lemma} \label{lem:Galois}
	There exists a compact $($in the complex topology$)$ subset $\cC^*$ of $\cC$, such that for all $c\in \cC_0$ of degree large enough, 
	at least half of the Galois conjugates of $c$ over $\K$ lie in $\cC^*$
\end{lemma}

\begin{proof}
	See \cite[Lemma 8.2]{MZ14a}.
\end{proof}

Note that, if $c\in \cC_0$, then all its Galois conjugates over $\K$ satisfy again some dependence relations, hence they must also lie in $\cC_0$.

We now cover the set $\cC^*$ appearing in Lemma~\ref{lem:Galois} with finitely many locally contractible compact subsets of $E^n$ which we call $D_1, \dots , D_{\gamma_1}$.

Let $D$ be one of these sets. We set $R=\mathrm{End}(E)$ and $P_j=(X_j,Y_j)$. For $\ol{a}=(a_1,\dots , a_m )\in \Z^m \setminus \{0\}$ and $\ol{b}=(b_1,\dots , b_n ) \in R^n \setminus \{0\}$ we set
\[ 
D(\ol{a},\ol{b}):=\left \{c\in D: ~ \prod_{i=1}^m Z_i(c)^{a_i}=1 \mbox{ and } \sum_{j=1}^n b_jP_j(c)=O 
\right \}.
\]

For the rest of the section the implied constants will depend on $\cC$ and $D$. Any further dependence will be expressed by an index.

\begin{lemma}
	If $c\in D \cap \cC_0$, there are $\ol{a} \in \Z^m \setminus \{0\}$ and $\ol{b}\in R^n \setminus \{0\}$ such that $c \in D(\ol{a},\ol{b})$ and
	\begin{equation}   \label{eq:abc}
	\max \{|\ol{a}|,|\ol{b}| \} \ll [\K(c):\K]^{\gamma_2},
	\end{equation}
	for some constant $\gamma_2>0$ depending on the curve $\cC$ and the set $D$, where $|\ol{a}| = \max \{|a_1| , \dots , |a_m|\}$ and $|\ol{b}| = \max \{|b_1| , \dots , |b_n|\}$.
\end{lemma}

\begin{proof}
	See \cite[Lemmas 5.1 and 5.2]{BC} and (if $E$ has CM) \cite[Lemma 6.1]{BCM}.
\end{proof}

We denote by $u_1,\dots, u_m$ the principal determinations of the standard logarithms  of $Z_1, \dots , Z_m$ and by $w_1, \dots , w_n$  the elliptic logarithms  of $P_1,\dots, P_n$ seen as analytic functions on (an open neighbourhood of) $D$. 
These functions satisfy the equations   
$$
u_i=p_i+2\pi \sqrt{-1} q_i ,  \quad   w_j=r_j +   s_j \tau, \ \  \mbox{for $i=1,\dots , m$ and $j=1,\dots , n$,}
$$
where $(1,\tau)$ is a basis of the period lattice of $E$ and $p_i, q_i, r_j, s_j$ are real-valued functions defined on $D$. If we view the compact set $D$ as subset of $\R^2$, we can define 
$$
\begin{array}{crcl}
\theta :& D\subset \R^2 & \rightarrow & \R^{2m+2n} \\
& c &\mapsto & (p_1(c), q_1(c),\dots, p_m(c), q_m(c), r_1(c), s_1(c),\dots , r_n(c),s_n(c)).
\end{array}
$$
The image $\theta(D)$ is a subanalytic surface of $\R^{2m + 2n}$ which we denote by $S$. This is a definable set in the o-minimal structure $\R_{\mathrm{an}}$. In this section definable means definable in $\R_{\mathrm{an}}$. Note that $\theta$ is injective. Moreover, as $D$ is compact, we have that the functions $q_i,r_j$ and $s_j$ take bounded values. 
The $p_1, q_1,\dots, p_m,q_m, r_1, s_1,\dots , r_n,s_n$ are sometimes called Betti-coordinates and $\theta$ the Betti-map.
\medskip

For any $b\in R$ and any point $P\in E(\Qbar)$, given $\rho+\sigma \tau$ an elliptic logarithm of $P$, 
 a logarithm of $bP$ is given by $\rho'+\sigma' \tau$, where 
$$
\left(\begin{array}{c}
\rho' \\ \sigma'
\end{array}  \right)=A(b)\left(\begin{array}{c}
\rho \\ \sigma
\end{array}  \right),
$$
for some $A(b)\in M_2(\Z)$, where $M_2(\Z)$ is the ring of $2 \times 2$ matrices over $\Z$. 

For $l \in\Z$, clearly
$$
A(l)=\left(\begin{array}{cc}
l& 0 \\ 0 &l
\end{array}  \right).
$$
If $R=\Z[\alpha]$ for some imaginary quadratic $\alpha$, we have that 
$$
A(l_1+\alpha l_2)=\left(\begin{array}{cc}
l_1& 0 \\ 0 &l_1
\end{array}  \right)+A(\alpha)\left(\begin{array}{cc}
l_2& 0 \\ 0 &l_2
\end{array}  \right).
$$

Note that, as the entries of $A(\alpha)$ are fixed and depend only on $E$, if $\abs{A}$ is the maximum of the absolute values of the entries of a matrix $A \in M_2(\Z)$, then we have $\abs{A(b)}\ll \abs{b}$ for all $b\in R$.

Using the function $\theta$ defined before, the points of $\cC_0$ that satisfy two relations will correspond to points of $S$ lying on linear varieties defined by equations of some special form with integer coefficients.
In particular, if  $c\in D(\ol{a},\ol{b})$, there are integers $e,f,g$ such that
$$
\begin{cases}
\sum_{i=1}^m a_i u_i=2\pi \sqrt{-1} e,  \\
\sum_{j=1}^n b_jw_j=f+ g\tau, 
\end{cases}
$$
which translates to
$$
\begin{cases}
\sum_{i=1}^m a_ip_i=0, \\
\sum_{i=1}^m a_iq_i=e,\\
\sum_{j=1}^n A(b_j) (r_j,s_j)^{\mathrm{t}}=(f,g)^{\mathrm{t}}, 
\end{cases}
$$
holding for $\theta(c)$ (here $\cdot ^{\mathrm{t}} $ denotes the transposition).

We define
\begin{multline*}
W=\left\lbrace(\alpha_1,\dots, \alpha_m,B_1, \dots , B_n, \sigma_1, \sigma_2, \sigma_3, p_1, q_1,\dots, p_m,q_m, r_1, s_1,\dots , r_n,s_n)\in \vphantom{\sum_{i=1}^m }  \right.  \\ \left. \R^m \times M_2(\R)^n\times \R^3\times S: \sum_{i=1}^m \alpha_i p_i=0  ,\sum_{i=1}^m \alpha_i q_i=\sigma_1 ,  \sum_{j=1}^n B_j(r_j,s_j)^{\mathrm{t}}=(\sigma_2,\sigma_3)^{\mathrm{t}}\right\rbrace.
\end{multline*}
This is a definable family of subsets of $\R^3\times S$ with parameter space $\R^m \times M_2(\R)^n$. For $\boldsymbol{\alpha}=(\alpha_1,\dots, \alpha_m )\in \R^m$ and $\ol{B}=(B_1, \dots , B_n)\in M_2(\R)^n$, we let
\begin{multline*}
W_{\boldsymbol{\alpha},\ol{B}}=\{ (\sigma_1, \sigma_2, \sigma_3, p_1, q_1,\dots, p_m,q_m, r_1, s_1,\dots , r_n,s_n )\in  \R^3\times S:  \\ (\alpha_1,\dots, \alpha_m,B_1, \dots , B_n, \sigma_1, \sigma_2, \sigma_3, p_1, \dots ,s_n ) \in W \}
\end{multline*}
be the fiber of $W$ above $(\boldsymbol{\alpha},\ol{B})$. Moreover, we let $\pi_1$ be the projection from $\R^3\times S \subseteq \R^3 \times \R^{2m+2n}$ to $\R^3$, while $\pi_2$ indicates the projection to $S$. We also define, for $T\geq 0$,
\begin{multline*}
W_{\boldsymbol{\alpha},\ol{B}}^\sim (\Q, T)=  \{ (\sigma_1, \sigma_2, \sigma_3, p_1, q_1,\dots, p_m,q_m, r_1, s_1,\dots , r_n,s_n   ) \in W_{\ol{\boldsymbol{\alpha}},\ol{B}}:  \\ (\sigma_1,\sigma_2,\sigma_3) \in \Q^3 \text{ and } H(\sigma_1,\sigma_2,\sigma_3)\leq T \},
\end{multline*}
where $H(\sigma_1,\sigma_2,\sigma_3)$ is the maximum of the absolute values of the numerators and denominators of the $\sigma_j$ when they are written in lowest terms.

Fix now $\ol{a}\in \Z^m$ and $\ol{b}\in R^n$.
Note that, if $c\in D(\ol{a}, \ol{b})$, then by the above discussion there are integers $e,f,g$ such that $(e,f,g, \theta (c))\in W_{\ol{a},A(\ol{b})}$, where $A(\ol{b})=(A(b_1),\dots, A(b_n))$. Since $  q_1,\dots,q_m, r_1, s_1,\dots ,r_n, s_n $ take bounded values as $D$ is compact, we can suppose that 
$$
\max \{|e|,|f|,|g|,|\ol{a}|,|A(\ol{b})| \}\leq T_0,
$$ 
for some $T_0$ with $T_0 \ll \max \{|\ol{a}|,|A(\ol{b})| \}\ll \max \{|\ol{a}|,|\ol{b}| \}$. 
Therefore, if we let 
$$
\Sigma_{\ol{a},\ol{b}} :=\pi_2^{-1} (\theta(D(\ol{a},\ol{b})))\cap W_{\ol{a},A(\ol{b})} ,
$$
then we have $\Sigma_{\ol{a},\ol{b}}  \subseteq   W_{\ol{a},A(\ol{b})} ^\sim (\Q, T_0) $. Note that $\theta(D(\ol{a},\ol{b}))\subseteq \pi_2(W_{\ol{a},A(\ol{b})})$.

We claim that, for every $\epsilon>0$, we have an upper bound for the cardinality of $D(\ol{a},\ol{b})$ of the form 
\begin{equation}   \label{eq:ab}
|D(\ol{a},\ol{b})|\ll_\epsilon (\max \{|\ol{a}|,|A(\ol{b})| \})^\epsilon.
\end{equation}
If not, by the previous considerations the following lemma would be contradicted.

\begin{lemma}
	For every $\epsilon>0$ we have $|\pi_2(\Sigma_{\ol{a},\ol{b}})|\ll_\epsilon T_0^\epsilon$. 
\end{lemma}

\begin{proof}
	Suppose there is a positive constant $\gamma_3=\gamma_3(W,\epsilon)$ such that $|\pi_2(\Sigma_{\ol{a},\ol{b}})|\geq \gamma_3 T_0^\epsilon$. Then, by \cite[Corollary 7.2]{HabPila16}, there exists a continuous and definable function $\delta:[0,1]\rightarrow W_{\ol{a},A(\ol{b})}$ such that
	\begin{enumerate}
		\item the map $\delta_1:=\pi_1 \circ \delta :[0,1]\rightarrow \R^3$ is semi-algebraic and its restriction to $(0,1)$ is real analytic;
		\item the composition $\delta_2:=\pi_2 \circ \delta :[0,1]\rightarrow S$ is non-constant;
		\item we have $\pi_2(\delta(0)) \in \pi_2(\Sigma_{\ol{a},\ol{b}})$.
	\end{enumerate}
	
	By rescaling and restricting the domain we can suppose that the path $\delta_1 $ is contained in a real algebraic curve. Moreover, by (3) above, there exists $c_0 \in D(\ol{a},\ol{b})$ with $\theta (c_0)=\delta_2(0)$.
	
	We now consider the map
	$$
	\begin{array}{clcl}
	\phi :& \Gm^m\times E^n & \rightarrow & \Gm \times E \\
	& (Z_1, \dots, Z_m,P_1, \dots , P_n) &\mapsto & (\prod_{i=1}^mZ_i^{a_i},\prod_{j=1}^n b_j P_j)
	\end{array}
	$$
	and its differential
	$$
	\begin{array}{clcl}
	d \phi :& \C^m\times \C^n & \rightarrow & \C\times\C \\
	& (u_1,\dots, u_m,w_1, \dots , w_n) &\mapsto & (\sum_{i=1}^m a_iu_i,\sum_{j=1}^n b_j w_j)=:(u',w'). 
	\end{array}
	$$
	Note that $\phi(\cC)$ cannot be constant, otherwise both conditions (1) and (2) in the hypotheses of Theorem \ref{thm:UI} would be false. Therefore  $\phi(\cC)$ is a curve.
	
	We can see $\sigma_1,\sigma_2, \sigma_3,  p_1, q_1,\dots, p_m,q_m, r_1, s_1,\dots , r_n,s_n$, and consequently $u_1,\dots, u_m$, $w_1, \dots , w_n,u', w'$, as coordinate functions on $[0,1]$. 
	We have that the transcendence degree trdeg$_\C \, \C(\sigma_1,\sigma_2,\sigma_3)\leq 1$ and recall that, by the definition of $W$, the two relations $u'=2\pi \sqrt{-1} \sigma_1$ and $w'= \sigma_2 +\sigma_3\tau$ must hold. We deduce that $$\trdeg_\C \, \C(\sigma_1,\sigma_2,\sigma_3,u',w')\leq 1.$$
	This gives a map $$\delta':=(u',w'): [0,1]\rightarrow \C \times \C$$ that is real semi-algebraic, continuous and with $\delta'  |_{(0,1)}$ real analytic. By Ax's Theorem \cite{Ax72} (see \cite[Theorem 5.4]{HabPila16}), the Zariski closure in $\Gm\times E$ of the image of $\exp \circ \delta'$, which is contained in $\phi(\cC)$, is a coset, that must actually be a torsion coset, because $\phi (c_0)$ is the neutral element of $\Gm\times E$. If this torsion coset is a curve, then it coincides with $\phi(\cC)$ and this contradicts the hypotheses of Theorem~\ref{thm:UI}. If the coset is a point, then $u'=\sum_{i=1}^m a_iu_i$ and $w'=\sum_{j=1}^n b_j w_j$ are both constant and equal to $d\phi(c_0)$ on $[0,1]$. This again contradicts the hypotheses of Theorem~\ref{thm:UI}.
\end{proof}

Now we are ready to prove Theorem~\ref{thm:UI}.

\begin{proof}[Proof of Theorem~\ref{thm:UI}]
Fix a $c_0\in \cC_0$ of large degree over $\K$. By Lemma~\ref{lem:Galois}  we have that one of the sets $D_1,\dots, D_{\gamma_1}$, say $D_1$, 
contains at least $ [\K(c_0):\K]/(2\gamma_1)$ conjugates of $c_0$. Moreover, if $c_0\in D_1(\ol{a},\ol{b})$ for some $\ol{a}\in \Z^m$ and $\ol{b}\in R^n$, all of these conjugates belong to $D_1(\ol{a},\ol{b})$. Therefore, combining this with \eqref{eq:abc} and \eqref{eq:ab}, we get
$$
[\K(c_0):\K] \ll |D_1(\ol{a},\ol{b})|\ll_\epsilon (\max \{|\ol{a}|,|\ol{b}|\})^\epsilon \ll_{\epsilon} [\K(c_0):\K]^{\gamma_2 \epsilon}, 
$$
which, after choosing $\epsilon < 1/(2\gamma_2)  $, leads to a contradiction if $[\K(c_0):\K]$ is too large.
This completes the proof of Theorem~\ref{thm:UI}. 
\end{proof}

We now formulate and prove a corollary of Theorem~\ref{thm:UI}. Notice that Lemma \ref{lem:new2.7} is a special case of it.

The point of the corollary is that, if one only needs to consider relations over $\Z$ among the $P_j(c)$, one can relax the hypotheses and assume that the $P_j$ are linearly independent over $\Z$ and not over $\mathrm{End}(E)$. 
Note that the analogous fact does not hold in the setting of Lemma \ref{lem:Gala}. Indeed, two points that are generically dependent over $\mathrm{End}(E)$ but not over $\Z$ can specialize infinitely many times to two torsion points.

\begin{corollary}\label{cor:UI}
		Let $E$ be an elliptic curve defined over a number field $\K$ by  a Weierstrass equation, and let $m, n \ge 1$ be integers. 
	Let $\cC$ be an irreducible curve in $ \Gm^m\times E^n$, also defined over $\K$, with coordinates $(Z_1, \dots , Z_m,X_1,Y_1,\dots ,X_n,Y_n )$ such that $Z_1, \dots , Z_m $ are multiplicatively independent and the points $P_j:=(X_j,Y_j)$ are linearly independent over $\Z$. Suppose moreover that at least one of the following conditions holds:
	\begin{enumerate}
		\item $Z_1, \dots , Z_m $ are multiplicatively independent modulo constants;
		\item the points $P_j$ are linearly independent over $\mathrm{End}(E)$ modulo points in $E(\overline{\Q})$.
	\end{enumerate}	
	Then, there are at most finitely many  points $c\in \cC(\C)$ such that the $Z_i(c)$ are multiplicatively dependent and the $P_j(c)$ are linearly dependent over $\Z$. 
\end{corollary}

\begin{proof}
	It is clear that our claim follows directly from Theorem \ref{thm:UI} in case the points $P_1, \dots, P_n$ are linearly independent over $\mathrm{End}(E)$. We only need to consider the case in which $R:=\End(E)\neq \Z$ and the points $P_1, \ldots, P_n$ satisfy a linear dependence relation over $\mathrm{End}(E)$, but none over $\Z$. Note that this automatically means that condition (2) does not hold and therefore (1) is satisfied.
	
Set $$\Lambda=\lg(\rho_1, \dots, \rho_n)\in R^n:~ \sum_{j=1}^n \rho_j P_j=O\rg.$$
	
	Our hypothesis on the $P_j$ implies that $\Lambda\cap \Z^n=\{0\}$.
	This is a finitely generated $R$-submodule of $ R^n$ of some rank $n'$, $1\leq n'< n$.
	It is a well known fact (see, e.g., Lemma 2.3 of \cite{BC}) that the set
	$$
	\mathcal{L}(\Lambda)=\lg (Q_1, \dots, Q_n)\in E^n:~ \sum_{j=1}^n \rho_j Q_j=O \text{ for all } (\rho_1, \dots, \rho_n)\in \Lambda \rg
	$$
	defines an algebraic subgroup of $E^n$ of dimension $n-n'$. Moreover, there is a surjective and finite homomorphism of algebraic groups
	$$
	\phi : \mathcal{L}(\Lambda)  \rightarrow   E^{n-n'} .
	$$
	Then, our hypotheses imply that $\phi(P_1, \dots , P_n)$ gives an irreducible curve that does not lie in a proper algebraic subgroup of $E^{n-n'}$.
	
	Suppose there are infinitely many points $c \in \cC(\C)$ such that $Z_1(c), \ldots, Z_m(c)$ are multiplicatively dependent and $P_1(c),\dots,$ $P_n(c)$ are linearly  dependent over $\Z$. Then, every $(P_1(c),\dots, P_n(c))$ lies in a proper algebraic subgroup of $\mathcal{L}(\Lambda)$ and its image via $\phi$ in a proper algebraic subgroup of  $ E^{n-n'}$. 
	A contradiction arises by applying Theorem \ref{thm:UI} to the curve in $\Gm^m\times E^{n-n'}$ given by $(Z_1, \ldots, Z_m,\phi(P_1, \dots , P_n))$, concluding the proof.
\end{proof}

\section{Proofs of main results}
\label{sect:proofs}

\subsection{Proof of Theorem~\ref{thm:Fp-A}} 
For any non-zero integer vector $\pmb{k}= \left(k_1, \ldots, k_m\right) \in \Z^m$,  we define the rational function 
$$
\Omega_{\pmb{k}}(X) = \varphi_1(X)^{k_1} \cdots \varphi_m(X)^{k_m}.
$$
We write $\varphi_i = f_i/g_i$ with relatively prime  polynomials $f_i, g_i \in \Z[X]$, 
$i =1, \ldots, m$, and thus we have $\Omega_{\pmb{k}}(X)=F_{\pmb{k}}(X)/G_{\pmb{k}}(X)$
with polynomials $F_{\pmb{k}}(X), G_{\pmb{k}}(X)\in\Z[X]$ defined by
\begin{equation}  \label{eq:FG}
\begin{split}
F_{\pmb{k}}(X)& =\prod_{\substack{1\le i\le m\\ k_i>0}}f_i(X)^{k_i} 
\prod_{\substack{1\le i\le m\\ k_i<0}}g_i(X)^{-k_i},\\
G_{\pmb{k}}(X)&=\prod_{\substack{1\le i\le m\\ k_i<0}}f_i(X)^{-k_i} 
\prod_{\substack{1\le i\le m\\ k_i>0}}g_i(X)^{k_i}.
\end {split}
\end{equation}

Recall that $\cS_1 \subset \C$ is the set of all the elements $\alpha\in \C$ which are solutions to the system of equations
\begin{equation}   \label{eq:Omega}
\Omega_{\pmb{k}}(X)-1=\Omega_{\pmb{\ell}}(X)-1=0
\end{equation}
for some linearly independent vectors $\pmb{k},\pmb{\ell}  \in \Z^m$. In what follows, we will always tacitly assume the vectors  $\pmb{k},\pmb{\ell}$ to be linearly independent. 

Clearly,  if $\alpha \in \cS_1$, then every Galois conjugate of $\alpha$ over $\Q$ is also in $\cS_1$. 

By  Lemma~\ref{lem:eff_Maurin} the set $\cS_1$ is finite and we have 
\begin{equation}  \label{eq:S1}
\#\cS_1 \ll_{\bphi} 1,
\end{equation}
where the implied constant is effectively computable. 

Let $W_{\cS_1} \in \Z[X]$ be the product of all the irreducible polynomials (without multiplicity) having some $\alpha \in \cS_1$ as a root. 
Clearly, we have 
\begin{equation*}
\deg W_{\cS_1} = \# \cS_1. 
\end{equation*}
Define
$$
P_{\pmb{\ell}}=\frac{F_{\pmb{\ell}}-G_{\pmb{\ell}}}{\gcd(F_{\pmb{\ell}}-G_{\pmb{\ell}},(F_{\pmb{\ell}}-G_{\pmb{\ell}})')} \mand 
\tilde{P}_{\pmb{\ell}}=\frac{P_{\pmb{\ell}}}{\gcd(P_{\pmb{\ell}},W_{\cS_1})}\in \Z[X].
$$
Note that since the polynomial $P_{\pmb{\ell}}$ has only simple roots, we have $\gcd(\tilde{P}_{\pmb{\ell}}, W_{\cS_1})=1$.

Then,  the system of equations
\begin{equation}\label{eq:KL}
F_{\pmb{k}}(X)-G_{\pmb{k}}(X)=\tilde{P}_{\pmb{\ell}}(X)=0
\end{equation}
has no solution over $\C$. Our aim is to show now that this system has no solution over $\overline{\F}_p$ if we take $p$ large enough 
when 
$$
\pmb{k}\in\{0,\pm 1,\dots, \pm K\}^m \setminus \{\mathbf{0}\}, \qquad \pmb{\ell}\in\{0,\pm 1,\dots, \pm L\}^m \setminus \{\mathbf{0}\}.
$$

To do this, first notice that $F_{\pmb{k}}(X)-G_{\pmb{k}}(X)$ is not identically zero, since the rational functions $\varphi_1, \ldots, \varphi_m$ are multiplicatively independent by assumption. 

Moreover, suppose that $F_{\pmb{k}}(X)-G_{\pmb{k}}(X)=a\in \Z\setminus \{0\}$. 
 Then the system \eqref{eq:KL} may have a solution over $\overline{\F}_p$ if and only if $p \mid a$. 
 But, if we take $p$ larger than $H(F_{\pmb{k}}-G_{\pmb{k}})$ then this is not the case. By Lemma~\ref{lem:HeightPoly}  we have 
\begin{align*}
\max\{\h (F_{\pmb{k}} ), \h (G_{\pmb{k}} )   \}
& \le \sum_{i=1}^m |k_i| \max\{\h(f_i)+\deg f_i, \h(g_i)+ \deg g_i\} \\
& \ll_{\bphi} K , 
\end{align*}
which  implies 
\begin{equation}   \label{eq:hFG}
\h (F_{\pmb{k}}-G_{\pmb{k}}) \ll_{\bphi} K,
\end{equation}
hence, if $p\ge \exp(cK)$ for some constant $c$ depending only on $\bphi$, then the system \eqref{eq:KL} has no solution over $\overline{\F}_p$ as wanted. 

Now assume that $F_{\pmb{k}}(X)-G_{\pmb{k}}(X)$ is non-constant; we denote 
$$
R_{\pmb{k},\pmb{\ell}}=\Res(F_{\pmb{k}}(X)-G_{\pmb{k}}(X), \tilde{P}_{\pmb{\ell}}(X)  ), 
$$
which is non-zero. 
So, if $p>|R_{\pmb{k},\pmb{\ell}}|$, then $p\nmid R_{\pmb{k},\pmb{\ell}}$, and thus the system of equations~\eqref{eq:KL} has no solution over $\overline{\F}_p$. 

Therefore,  it is easy to see that the desired result follows when   
\begin{equation}\label{eq:KL_bound}
 p>\max\left \{ \exp(cK),\ \max_{\pmb{k}\in\{0,\pm 1,\dots, \pm K\}^m\setminus\{\mathbf{0}\}}\max_{\pmb{\ell}\in\{0,\pm 1,\dots, \pm L\}^m\setminus\{\mathbf{0}\}}| R_{\pmb{k},\pmb{\ell}} | \right \}.
\end{equation}
Hence, it remains to estimate $R_{\pmb{k},\pmb{\ell}}$, for the parameters $\pmb{k}$ and $\pmb{\ell}$ in
the same ranges as on the right hand side of~\eqref{eq:KL_bound}.

We note that 
considering $F_{\pmb{k}}$ and $G_{\pmb{k}}$, defined by~\eqref{eq:FG}, as  products of at most $|k_1| + \cdots + |k_m|$ polynomials, we have
$$
\deg(F_{\pmb{k}}-G_{\pmb{k}})\ll_{\bphi} K,
$$
and 
\[
\h (F_{\pmb{k}}-G_{\pmb{k}}) \ll_{\bphi} K
\]
as shown previously.

For the polynomial $\tilde{P}_{\pmb{\ell}}$, we clearly have
$$
\deg(\tilde{P}_{\pmb{\ell}})\le \deg(F_{\pmb{\ell}}-G_{\pmb{\ell}})\ll_{\bphi} L,
$$
and  it follows from~\eqref{eq:Mahler} that
\begin{equation*} 
\begin{split}
 H(\tilde{P}_{\pmb{\ell}})&\leq 2^{\deg \tilde{P}_{\pmb{\ell}}}M(\tilde{P}_{\pmb{\ell}})
 \leq 2^{\deg (F_{\pmb{\ell}}-G_{\pmb{\ell}})}M(F_{\pmb{\ell}}-G_{\pmb{\ell}})\\
 &\leq 2^{\deg (F_{\pmb{\ell}}-G_{\pmb{\ell}})}\sqrt{\deg (F_{\pmb{\ell}}-G_{\pmb{\ell}})+1  }H(F_{\pmb{\ell}}-G_{\pmb{\ell}}).
\end{split}
 \end{equation*}
 Thus, we conclude that
 $$
 \h(\tilde{P}_{\pmb{\ell}})\ll_{\bphi} L.
 $$
 
 Therefore,  for the parameters $\pmb{k}$ and $\pmb{\ell}$ in
the same ranges as on the right hand side of~\eqref{eq:KL_bound},  by Lemma~\ref{lem:res} we obtain
\begin{equation}
\label{eq:Bound R}
\log |R_{\pmb{k},\pmb{\ell}} | \ll_{\bphi} KL ,
\end{equation}
which together with \eqref{eq:KL_bound} gives the desired lower bound $\exp(c_1 KL)$ for $p$, for some constant $c_1$ depending only on $\bphi$. 
Since all the implied constants in the above estimates are effectively computable, the constant $c_1$ is also effectively computable. 

Finally, by the above discussions,  when $p > \exp(c_1 KL)$, the system of equations~\eqref{eq:KL} has no solution 
over $\ov{\F}_p$ for any linearly independent vectors $\pmb{k},\pmb{\ell}$ in the same ranges as on the right hand side of~\eqref{eq:KL_bound}. 
In addition, if $\alpha$ is a solution of \eqref{eq:Omega} but not a solution of \eqref{eq:KL} over $\ov{\F}_p$, 
then $\alpha$ must be a root of $W_{\cS_1}$ over $\ov{\F}_p$. 
Hence, for the set $\cA_{\bphi}(p,K,L)$ we have 
$$
\#\cA_{\bphi}(p,K,L) \le \deg W_{\cS_1} = \# \cS_1,
$$
which completes the proof.

\subsection{Proof of Corollary~\ref{cor:Fp-multdep}}
This follows directly from Theorem~\ref{thm:Fp-A} applied to the rational functions $\varphi_1,\ldots,\varphi_m, \varrho_1, \ldots, \varrho_n$ instead of $\varphi_1,\ldots,\varphi_m$.

\subsection{Proof of Theorem~\ref{thm:Ord2}}
We use the notation in the proof of Theorem~\ref{thm:Fp-A}, with $\pmb{k},\pmb{\ell}$ linearly independent. Define  
$$
T = \prod_{\pmb{k} \in \{0,\pm 1,\ldots, \pm K\}^m \setminus \{\vec{0}\}}\,   \prod_{\pmb{\ell}\in\{0,\pm 1,\dots, \pm L\}^m\setminus\{\mathbf{0}\}} |R_{\pmb{k},\pmb{\ell}} |.
$$
Then, by \eqref{eq:Bound R} we have 
$$
\log T \ll_{\bphi} (KL)^{m+1}.
$$
For $N$ large enough, we take 
$$
K =  L =  \rf{(N/\log N)^{1/(2m+2)}}. 
$$
Since $T$ has at most $c\log T / \log\log T$ distinct prime divisors for some absolute constant $c$, 
we derive that $v_p(T) = 0$ for all but at most 
$$
c\log T / \log\log T  \ll_{\bphi} (KL)^{m+1} / \log((KL)^{m+1}) \ll N(\log N)^{-2}
$$
primes $p$ (even without the restriction $p \le N$). 

In addition, we define  
$$
J = \prod_{\pmb{k} \in \{0,\pm 1,\ldots, \pm K\}^m \setminus \{\vec{0}\}} H(F_{\pmb{k}}-G_{\pmb{k}}). 
$$
By \eqref{eq:hFG}, we have 
$$
\log J \ll_{\bphi} K^{m+1}.
$$
As the above, we obtain that $v_p(J) = 0$ for all but at most 
$$
c\log J / \log\log J  \ll_{\bphi}  N^{1/2}(\log N)^{-3/2}
$$
primes $p$ (even without the restriction $p \le N$).

Now, consider primes $p$ with $v_p(J) = 0$ and $v_p(T) = 0$ and vectors $\pmb{k}, \pmb{\ell}$ mentioned above. 
If $F_{\pmb{k}}-G_{\pmb{k}}$ is a non-zero constant, 
since $v_p(J) = 0$ we have that $p$ does not divide $F_{\pmb{k}}-G_{\pmb{k}}$, 
and thus the system of equations~\eqref{eq:KL} has no solution over $\ov{\F}_p$.  
If   $F_{\pmb{k}}-G_{\pmb{k}}$ is non-constant, 
then since $v_p(T) = 0$ we obtain that $p$ does not divide the resultant $R_{\pmb{k},\pmb{\ell}}$, 
and so again this system of equations  has no solution over $\ov{\F}_p$. 
Hence, as before we have that the cardinality $\#\cA_{\bphi}(p,K,L)$ is at most $c_2$ 
which is a constant depending only on $\bphi$. 
Note that if $\alpha \not\in \cA_{\bphi}(p,K,L)$, then the elements $\varphi_1(\alpha), \ldots, \varphi_m(\alpha)$ 
do not satisfy two independent multiplicative relations with exponents bounded above by $K = L$ in absolute value. 
Hence, at least $m-1$ elements of $\varphi_1(\alpha), \ldots, \varphi_m(\alpha)$ are of order at least $K$. 
This completes the proof.

\subsection{Proof of Theorem~\ref{thm:Ord-A}} 
We follow the approach in proving Theorem~\ref{thm:Ord2}. 
However, this time for $N$ large enough, we take 
$$
K =  \rf{N^{n/(2mn+m+n)}/(\log N)^{1/(2m)}} \mand L=  \rf{N^{m/(2mn+m+n)}/(\log N)^{1/(2n)}}.
$$

Note that any $K$-multiplicatively  independent elements  $\alpha_1,\ldots, \alpha_m\in \ov \F_p^*$  
generate a subgroup of $\ov{\F}_p^*$ of order at least $K^m \ge N^{mn/(2mn+m+n)}(\log N)^{-1/2}$. 
In addition, we have  $L^n \ge N^{mn/(2mn+m+n)}(\log N)^{-1/2}$.  
Then, the desired result follows similarly.

\subsection{Proof of Theorem~\ref{thm:Fp-multdep A}}
We proceed as in the proof of Theorem~\ref{thm:Fp-A}. In what follows, we will always suppose the vectors $\pmb{k}$ and $\pmb{\ell}$ to be linearly independent.
However, this time we use 
Lemma~\ref{lem:Res} to estimate the number $s(\pmb{k},\pmb{\ell})$ of solutions  to 
the system of equations~\eqref{eq:KL} over $\ov{\F}_p$  for the parameters
$\pmb{k}$ and $\pmb{\ell}$ in the same ranges as on the right hand side of~\eqref{eq:KL_bound}. 
Besides, according to the condition in Lemma~\ref{lem:Res} that at least one of the polynomials in the system~\eqref{eq:KL} should not be identically zero modulo $p$, the prime $p$ should be large enough. 
For this, choosing $p$ larger than the absolute value of the discriminant of the square-free part of the polynomial $f_1g_1 \ldots f_m g_m$ 
(that is, distinct roots of this polynomial in $\overline{\Q}$ remain distinct modulo $p$) 
and noticing that $\varphi_1,\ldots,\varphi_m$ are multiplicatively independent modulo constants, 
we have that $\varphi_1, \ldots, \varphi_m$ are multiplicatively independent viewed as rational functions in $\ov{\F}_p(X)$, 
and so $F_{\pmb{k}}(X)-G_{\pmb{k}}(X)$ does not vanish modulo $p$.

Now, using Lemma~\ref{lem:Res} we obtain
$$
\sum_{\substack{\pmb{k} \in \{0,\pm 1,\ldots, \pm K\}^m \setminus \{\vec{0}\} \\ \pmb{\ell}\in\{0,\pm 1,\dots, \pm L\}^m\setminus\{\mathbf{0}\}}} s(\pmb{k},\pmb{\ell})  \le  v_p(T), 
$$
where 
$$
T = \prod_{\pmb{k} \in \{0,\pm 1,\ldots, \pm K\}^m \setminus \{\vec{0}\}}\,   \prod_{\pmb{\ell}\in\{0,\pm 1,\dots, \pm L\}^m\setminus\{\mathbf{0}\}} |R_{\pmb{k},\pmb{\ell}} |.
$$

Using the bound~\eqref{eq:Bound R},  we get the desired upper bound for $\log T$. 
Hence, we have 
$$
\#\cA_{\bphi}(p,K,L) \le v_p(T) + \deg W_{\cS_1} = v_p(T) + \# \cS_1, 
$$
which completes the proof by noticing \eqref{eq:S1}.

\subsection{Proof of Theorem~\ref{thm:Fp-multdep D}}

It suffices to follow the same arguments as in the proof of Theorem~\ref{thm:Fp-multdep A} by noticing that in this case we need to consider vectors 
$\pmb{k} \in \{0,\pm 1,\ldots, \pm K\}^m \setminus \{\vec{0}\}$ and $\pmb{\ell}\in\{0,\pm 1,\dots, \pm L\}^n\setminus\{\mathbf{0}\}$, 
and thus the contribution of those $\pmb{\ell}$ is $L^{n+1}$ in the bound of $\log T$ instead of $L^{m+1}$. 
This completes the proof.

\subsection{Proof of Theorem \ref{thm:EC}}
For any $m$-tuple $\pmb{k}=(k_1,\dots, k_m)\in\Z^m\setminus \{\mathbf{0}\}$ define 
$$
\Omega_{\pmb{k}}=\varphi_1^{k_1}\dots \varphi_m^{k_m}\in \Q(X), 
$$
and for any $n$-tuple $\pmb{\ell}=(\ell_1,\dots, \ell_n)\in\Z^n\setminus \{\mathbf{0}\}$ let $\Theta_{\pmb{\ell}}\in\Q(X)$ be defined by 
\begin{equation*}
\Theta_{\pmb{\ell}}=\left\{\begin{array}{ll}\Psi_{\ell_1}\circ \varrho_1& \textrm{if $n=1$,}\\
\sigma_n\left(\frac{\phi_{\ell_1}}{\psi_{\ell_1}^2}\circ \varrho_1,\dots, \frac{\phi_{\ell_n}}{\psi_{\ell_n}^2}\circ \varrho_n \right) & \textrm{if $n \ge 2$,}\end{array}\right.
\end{equation*}
where $\psi_{\ell_i},\Psi_{\ell_1}, \phi_{\ell_i}$ have been defined in Section~\ref{sec:div-poly} and $\sigma_n$ is the $n$-th summation polynomial associated to the curve~$E$ defined in Lemma~\ref{lem:summation-polys}. 
By assumption, $\Theta_{\pmb{\ell}}$ is non-zero. 

We remark that for any $\alpha \in \cB_{\bphi,\brho,E}(p,K,L)$, if $(\varrho_i(\alpha), \cdot)$ is a torsion point for some $1 \le i \le n$, 
then this situation is essentially reduced to the case when $n=1$.     

So, from now on, when $n \ge 2$, we do not consider those $\alpha \in \cB_{\bphi,\brho,E}(p,K,L)$ such that $(\varrho_i(\alpha), \cdot)$ is a torsion point for some $1 \le i \le n$.

The proof follows similar lines as in the proof of Theorem~\ref{thm:Fp-A}. Indeed, recall that $\cS_2 \subset \C$ is the set of all the elements $\alpha\in \C$ which are solutions to the system of equations
$$
\Omega_{\pmb{k}}(X)-1=\Theta_{\pmb{\ell}}(X)=0 \quad \text{for some } \pmb{k}\in\Z^m \setminus \{\mathbf{0}\} \text{ and } \pmb{\ell}\in\Z^n \setminus \{\mathbf{0}\}.
$$
By Lemma~\ref{lem:new2.7} (for $n=1$ it suffices to apply Lemma~\ref{lem:tors}) and noticing \eqref{eq:nP} and the definition of summation polynomials, 
the set $\cS_2$ is finite and we have 
$$
\#\cS_2 \ll_{\bphi, \brho, E} 1, 
$$ 
where the implied constant is effectively computable when $n=1$ by Lemma~\ref{lem:tors}. 

Write
$$
\Omega_{\pmb{k}}=\frac{F_{\pmb{k}}}{G_{\pmb{k}}},\quad \gcd(F_{\pmb{k}},G_{\pmb{k}})=1, \mand \Theta_{\pmb{\ell}}=\frac{U_{\pmb{\ell}}}{V_{\pmb{\ell}}},\quad \gcd(U_{\pmb{\ell}},V_{\pmb{\ell}})=1,
$$
with polynomials $F_{\pmb{k}},G_{\pmb{k}},U_{\pmb{\ell}}, V_{\pmb{\ell}}\in \Z[X]$. 

Let $W_{\cS_2} \in \Z[X]$ be the product of all the irreducible polynomials (without multiplicity) having some $\alpha \in \cS_2$ as a root. 
Define
$$
\overline{U}_{\pmb{\ell}}=\frac{U_{\pmb{\ell}}}{\gcd(U_{\pmb{\ell}},U_{\pmb{\ell}}')} \mand 
\tilde{U}_{\pmb{\ell}}=\frac{\overline{U}_{\pmb{\ell}}}{\gcd(\overline{U}_{\pmb{\ell}},W_{\cS_2})}\in \Z[X].
$$
Note that since the polynomial $\overline{U}_{\pmb{\ell}}$ has only simple roots, we have $\gcd(\tilde{U}_{\pmb{\ell}}, W_{\cS_2})=1$.

Then,  the system of equations
\begin{equation}\label{eq:OT}
F_{\pmb{k}}(X)-G_{\pmb{k}}(X)=\tilde{U}_{\pmb{\ell}}(X)=0
\end{equation}
has no solution over $\C$. We want to show that, if $p$ is large enough, then this system has no solution over $\overline{\F}_p$ 
when
$$
\pmb{k}\in\{0,\pm 1,\dots, \pm K\}^m \setminus \{\mathbf{0}\}, \qquad \pmb{\ell}\in\{0,\pm 1,\dots, \pm L\}^n \setminus \{\mathbf{0}\}.
$$

If $F_{\pmb{k}}(X)-G_{\pmb{k}}(X)$ is a non-zero constant, using the same argument as in Theorem \ref{thm:Fp-A}, we have that if $p\ge \exp(cK)$ for some constant $c$ depending only on $\pmb{\varphi}$, then $p\nmid F_{\pmb{k}}(X)-G_{\pmb{k}}(X)$ and the above system has no solution over $\overline{\F}_p$. 

Now, let us assume that $F_{\pmb{k}}(X)-G_{\pmb{k}}(X)$ is not constant and denote 
$$
R_{\pmb{k},\pmb{\ell}}=\Res(F_{\pmb{k}}(X)-G_{\pmb{k}}(X), \tilde{U}_{\pmb{\ell}}(X)  ),
$$
which is non-zero. 
So, if $p>|R_{\pmb{k},\pmb{\ell}}|$, then $p\nmid R_{\pmb{k},\pmb{\ell}}$, and so the system of equations~\eqref{eq:OT} has no solution over $\overline{\F}_p$. 

Therefore, it is easy to see that the desired result follows when 
\begin{equation}\label{eq:OT_bound}
 p>\max \left \{ \exp(cK),\,\max_{\pmb{k}\in\{0,\pm 1,\dots, \pm K\}^m\setminus\{\mathbf{0}\}}\max_{\pmb{\ell}\in\{0,\pm 1,\dots, \pm L\}^n\setminus\{\mathbf{0}\}}| R_{\pmb{k},\pmb{\ell}} | \right \}.
\end{equation}
Hence, it remains to estimate $R_{\pmb{k},\pmb{\ell}}$, for the parameters $\pmb{k}$ and $\pmb{\ell}$ in
the same ranges as on the right hand side of~\eqref{eq:OT_bound}.

From the proof of Theorem~\ref{thm:Fp-A}, for any $\pmb{k}\in\{0,\pm 1,\dots, \pm K\}^m\setminus\{\mathbf{0}\}$ we have
\begin{equation}  \label{eq:dh-FG}
\deg (F_{\pmb{k}}-G_{\pmb{k}}) \ll_{\bphi} K, \qquad \h(F_{\pmb{k}}-G_{\pmb{k}})\ll_{\bphi} K. 
\end{equation}

If $n=1$ (that is, $\pmb{\ell}=\ell_1$), then directly by Lemmas~\ref{lem:hight-composition} and~\ref{lem:Psi-height}  and by \eqref{eq:div-degree}, for $\ell_1\le L$, we have
\begin{equation}\label{eq:n1}
\deg \Theta_{\ell_1} \ll_{\varrho_1} L^2 \mand  \h(\Theta_{\ell_1})\ll_{\varrho_1, E} L^2. 
 \end{equation}

Now, we assume that $n \ge 2$. In this case, it follows from~\eqref{eq:Mahler} that
\begin{equation}\label{eq:h(U)}
\begin{split}
 H(\tilde{U}_{\pmb{\ell}})&\leq 2^{\deg \tilde{U}_{\pmb{\ell}}}M(\tilde{U}_{\pmb{\ell}})
 \leq 2^{\deg U_{\pmb{\ell}}}M(U_{\pmb{\ell}})\\
 &\leq 2^{\deg U_{\pmb{\ell}}}\sqrt{\deg U_{\pmb{\ell}}+1  }H(U_{\pmb{\ell}}).
\end{split}
 \end{equation}
By Lemmas~\ref{lem:hight-composition},~\ref{lem:Psi-height} and \ref{lem:division} we have
\begin{equation}\label{eq:height-of-comp}
\begin{split}
 \deg(\phi_{\ell_i}\circ \varrho_i) \ll_{\varrho_i} \ell_i^{2}&\mand \deg(\psi^2_{\ell_i}\circ \varrho_i) \ll_{\varrho_i} \ell_i^{2},\\
 \h(\phi_{\ell_i}\circ \varrho_i) \ll_{\varrho_i, E} \ell_i^{2}&\mand \h(\psi^2_{\ell_i}\circ \varrho_i) \ll_{\varrho_i, E} \ell_i^{2}. 
 \end{split}
\end{equation}
Applying again Lemma~\ref{lem:hight-composition} (with $R=\sigma_n$ and $f_i=\frac{\phi_{\ell_i}}{\psi_{\ell_i}^2}\circ \varrho_i$, $i=1,\ldots,n$), Lemmas~\ref{lem:summation-polys} and~\ref{lem:height_of_summation} and the estimates~\eqref{eq:height-of-comp}, for $\pmb{\ell}\in\{0,\pm 1,\dots, \pm L\}^n\setminus\{\mathbf{0}\}$ we obtain
\begin{align}
\label{eq:h(Theta)}
\deg \Theta_{\pmb{\ell}} \ll_{\brho} L^2, \qquad \h(\Theta_{\pmb{\ell}})\ll_{\brho, E} L^2.
\end{align}

Now, since by definitions of degree and height of a rational function we have
$$
\deg \tilde{U}_{\pmb{\ell}} \le \deg U_{\pmb{\ell}} \le \deg \Theta_{\pmb{\ell}}, \qquad 
\h(U_{\pmb{\ell}})\le \h(\Theta_{\pmb{\ell}}),
$$
from~\eqref{eq:n1},~\eqref{eq:h(U)} and~\eqref{eq:h(Theta)} for both cases we conclude that
\begin{equation}  \label{eq:h-Un}
\deg \tilde{U}_{\pmb{\ell}} \ll_{\brho} L^2, \qquad 
h(\tilde{U}_{\pmb{\ell}}) \ll_{\brho, E} L^2. 
\end{equation}

Therefore,  for the parameters $\pmb{k}$ and $\pmb{\ell}$ in
the same ranges as on the right hand side of~\eqref{eq:OT_bound},  
by Lemma~\ref{lem:res} and using \eqref{eq:dh-FG} and \eqref{eq:h-Un},  we obtain
\begin{equation}
\label{eq:Bound R2}
\log |R_{\pmb{k},\pmb{\ell}}| \ll_{\bphi, \brho, E} K L^2, 
\end{equation}
which together with \eqref{eq:OT_bound} gives the desired lower bound $\exp(c_1 KL^2)$ for $p$,  for some constant $c_1$ depending only on $\pmb{\varphi}, \brho$ and $E$. 
Since the implied constants in \eqref{eq:dh-FG}, \eqref{eq:n1}, \eqref{eq:height-of-comp}, \eqref{eq:h(Theta)} and \eqref{eq:h-Un} 
are all effectively computable, the constant $c_1$ is also effectively computable. 

Finally, by the above discussions,  when $p > \exp(c_1 KL^2)$, the system of equations~\eqref{eq:OT} has no solution 
over $\ov{\F}_p$ for any $\pmb{k}, \pmb{\ell}$ in the same ranges as on the right hand side of~\eqref{eq:OT_bound}.
Hence, as before, we obtain 
$$
\#\cB_{\bphi, \brho, E}(p,K,L) \le \# \cS_2.
$$  
This completes the proof.

\subsection{Proof of Theorem~\ref{thm:Ord-B}}
We follow the approach in the proof of Theorem~\ref{thm:Ord2} by using \eqref{eq:Bound R2} instead of \eqref{eq:Bound R} 
and considering the system of equations \eqref{eq:OT} and  the set $\cB_{\bphi, \brho, E}(p,K,L)$ instead of 
\eqref{eq:KL} and  the set $\cA_{\bphi}(p,K,L)$. 
This time, for $N$ large enough, take 
$$
K =  \rf{N^{n/(2mn +2m+n)}/(\log N)^{1/(2m)}} \mand L =  \rf{N^{m/(2mn +2m+n)}/(\log N)^{1/(2n)}}. 
$$
Then, one can obtain the desired result similarly.

\subsection{Proof of Theorem~\ref{thm:EC-multdep B}}
We proceed as in the proof of Theorem~\ref{thm:EC} and follow the approach in proving Theorem~\ref{thm:Fp-multdep A}. 
However, this time we use the bound \eqref{eq:Bound R2} instead of \eqref{eq:Bound R}. 
So, the cardinality $\#\cB_{\bphi, \brho, E}(p,K,L)$ is at most $v_p(T) +\#\cS_2$ 
when $p$ is greater than an effectively computable constant $c_2$ depending only on $\bphi, \brho$ and $E$. 
The desired result then follows.

\subsection{Proof of Theorem~\ref{thm:EC-C}}
For any $n$-tuple $\pmb{k}=(k_1,\dots, k_n) \in \Z^n\setminus \{\mathbf{0}\}$, 
as before we define 
\begin{equation*}
\Theta_{\pmb{k}}=\left\{\begin{array}{ll}\Psi_{k_1}\circ \varrho_1& \textrm{if $n=1$,}\\
\sigma_n\left(\frac{\phi_{k_1}}{\psi_{k_1}^2}\circ \varrho_1,\dots, \frac{\phi_{k_n}}{\psi_{k_n}^2}\circ \varrho_n \right) & \textrm{if $n \ge 2$.}\end{array}\right.
\end{equation*}
By assumption, it is non-zero.

The proof follows similar lines as in the proof of Theorem~\ref{thm:EC}.

Recall that $\cS_3 \subset \C$ is the set of all the elements $\alpha\in \C$ which are solutions to the system of equations
\begin{equation*}   \label{eq:Theta}
\Theta_{\pmb{k}}(X) = \Theta_{\pmb{\ell}}(X) = 0
\end{equation*}
for some linearly independent vectors $\pmb{k},\pmb{\ell}  \in \Z^n$. In what follows, we will always assume the vectors $\pmb{k},\pmb{\ell}$ to be linearly independent. 
Note that $\pmb{k},\pmb{\ell}$ are also linearly independent over $\mathrm{End}(E)$. 
Then, by  Lemma~\ref{lem:Gala} the set $\cS_3$ is finite and we have 
\begin{equation*}  \label{eq:S3}
\#\cS_3 \ll_{\brho,E} 1. 
\end{equation*}

Let $W_{\cS_3} \in \Z[X]$ be the product of all the irreducible polynomials (without multiplicity) having some $\alpha \in \cS_3$ as a root. 
As in the proof of Theorem~\ref{thm:EC}, we define the polynomials $U_{\pmb{k}}$ and $\tilde{U}_{\pmb{\ell}}$ 
with respect to $\Theta_{\pmb{k}}$ and $\Theta_{\pmb{\ell}}$ by using $W_{\cS_3}$ (instead of $W_{\cS_2}$). 

Then,  the system of equations   
\begin{equation}\label{eq:Ukl}
U_{\pmb{k}}(X) = \tilde{U}_{\pmb{\ell}}(X)=0
\end{equation}
has no solution over $\C$. Our aim is to show that, for $p$ large enough, this system has no solution over $\overline{\F}_p$ 
when  
$$
\pmb{k}\in\{0,\pm 1,\dots, \pm K\}^n \setminus \{\mathbf{0}\}, \qquad \pmb{\ell}\in\{0,\pm 1,\dots, \pm L\}^n \setminus \{\mathbf{0}\}.
$$

First, using the estimate \eqref{eq:h(Theta)}  (by replacing $L$ with $K$ there) and noticing $U_{\pmb{k}}$ is the numerator of $\Theta_{\pmb{k}}$, 
we have 
\begin{equation}  \label{eq:hUk}
\h(U_{\pmb{k}}) \ll_{\brho, E} K^2. 
\end{equation}

If $U_{\pmb{k}}(X)$  is a non-zero constant,    using \eqref{eq:hUk} we obtain that there exists a constant $c$ depending only on $\pmb{\rho}$ and $E$, such that, if $p>\exp(cK^2)$, then $p$ does not divide $U_{\pmb{k}}(X)$, and the system \eqref{eq:Ukl} has no solution over $\overline{\F}_p$.

Now, let us assume that $U_{\pmb{k}}(X)$ is non-constant and denote
$$
R_{\pmb{k},\pmb{\ell}}=\Res(U_{\pmb{k}}(X), \tilde{U}_{\pmb{\ell}}(X) ),
$$
which is non-zero. 
So, if $p>|R_{\pmb{k},\pmb{\ell}}|$, then $p\nmid R_{\pmb{k},\pmb{\ell}}$, and so the system of equations~\eqref{eq:Ukl} has no solution over $\overline{\F}_p$. 
It remains to estimate $R_{\pmb{k},\pmb{\ell}}$.   
Applying similar lines as in the proof of Theorem~\ref{thm:EC}, 
we obtain
\begin{equation}   \label{eq:logRkl}
\log |R_{\pmb{k},\pmb{\ell}}| \ll_{\brho, E} K^2 L^2. 
\end{equation}

Therefore, we obtain that there exists an effectively computable constant  $c_1$ 
depending only on  $\brho, E$ such that for any prime $p>\exp (c_1 K^2 L^2)$, 
the system of equations~\eqref{eq:Ukl} has no solution 
over $\ov{\F}_p$ for any $\pmb{k}, \pmb{\ell}$ in the above ranges, 
and so as before, for the set \eqref{eq:Set C} we have
$$
\#\cC_{\brho, E}(p,K,L)  \leq \#\cS_3.
$$

\subsection{Proof of Theorem~\ref{thm:EC-Ord2}}   
We follow the approach in the proof of Theorem~\ref{thm:Ord2} by using \eqref{eq:hUk} and \eqref{eq:logRkl} instead of \eqref{eq:hFG} and \eqref{eq:Bound R} 
and by considering the system of equations \eqref{eq:Ukl} and  the set $\cC_{\brho, E}(p,K,L)$ instead of 
\eqref{eq:KL} and  the set $\cA_{\bphi}(p,K,L)$. 
This time, for $N$ large enough, take 
$$
K = L =   \rf{(N/\log N)^{1/(2n+4)}}. 
$$
Then, one can obtain the desired result similarly.

\subsection{Proof of Theorem~\ref{thm:EC-Ord3}}
We obtain the desired result by 
 following the approach in proving Theorem~\ref{thm:EC-Ord2} with
$$
K =  \rf{N^{n/(2mn +2m+2n)}/(\log N)^{1/(2m)}} \mand L =  \rf{N^{m/(2mn +2m+2n)}/(\log N)^{1/(2n)}}. 
$$

\section*{Acknowledgement} 
The authors are grateful to Gabriel Dill, Igor Shparlinski and Umberto Zannier for helpful discussions, to the authors of~\cite{KMS} for sending them a preliminary version of their work, and to the referee for far valuable comments and suggestions.
For this research, L.M. was supported by the
Austrian Science Fund (FWF): Project P31762, A.O.  was supported by the Australian Research Council Grants DP180100201 and DP200100355, 
and M.S. was partially supported by the Australian Research Council Grant DE190100888. A.O. also gratefully acknowledges the generosity and hospitality of the Max Planck Institute for Mathematics where parts of her work on this project were developed.

\end{document}